\documentclass{amsart}
\usepackage{graphicx}
\usepackage{geometry}
\usepackage[utf8]{inputenc}
\usepackage[T1]{fontenc}
\usepackage{amsthm}
\usepackage{amsmath}
\usepackage{amsfonts}
\usepackage{amssymb}
\usepackage{graphicx}
\usepackage{xcolor}
\usepackage{yfonts}
\usepackage{tkz-fct}
\usepackage{booktabs}
\usepackage{setspace}
\usepackage{tikz-cd}
\usetikzlibrary{cd}
\usetikzlibrary{arrows}
\usepackage{quiver}
\usepackage{hyperref}
\usepackage[capitalise]{cleveref}

\newcommand{\comm}[1]{}
\newtheorem{Theorem}{Theorem}[section]

\newtheorem{Corollary}[Theorem]{Corollary}
\newtheorem{Lemma}[Theorem]{Lemma}

\newtheorem{Question}[Theorem]{Question}

\newtheorem{Proposition}[Theorem]{Proposition}
\newtheorem{Definition}[Theorem]{Definition}
\newtheorem{Remark}[Theorem]{Remark}

\numberwithin{equation}{section}

\DeclareMathOperator{\qu}{Quot_d(\mathcal{O}_{\mathbb{A}^{n}}^{\oplus \textit{r}})}
\DeclareMathOperator{\aff}{\mathbb{A}}
\DeclareMathOperator{\Gm}{\mathbb{G}_m}
\DeclareMathOperator{\Hilba}{Hilb_2(\mathbb{A}^n \times \mathbb{P}^{\textit{r}-1})}
\DeclareMathOperator{\qudwa}{Quot_2(\mathcal{O}_{\mathbb{A}^n}^{\oplus  \textit{r}})}
\DeclareMathOperator{\Qnrd}{Quot_d^{tr=0}(\mathcal{O}_{\mathbb{A}^{n}}^{\oplus \textit{r}})}
\DeclareMathOperator{\Qnrdin}{Quot_d^{tr=0}(\mathcal{O}_{\mathbb{A}^{\infty}}^{\oplus \textit{r}})}
\DeclareMathOperator{\qudwdw}{Quot_2(\mathcal{O}_{\mathbb{A}^2}^{\oplus  \textit{r}})}
\DeclareMathOperator{\quddw}{Quot_d(\mathcal{O}_{\mathbb{A}^2}^{\oplus  \textit{r}})}
\DeclareMathOperator{\qin}{Quot_d(\mathcal{O}_{\mathbb{A}^{\infty}}^{\oplus \textit{r}})}
\DeclareMathOperator{\qupludw}{Quot_d(\mathcal{O}_{\mathbb{A}^{n+2}}^{\oplus \textit{r}})}
\DeclareMathOperator{\quplu}{Quot_d(\mathcal{O}_{\mathbb{A}^{n+1}}^{\oplus \textit{r}})}

\DeclareMathOperator{\Yhat}{Y_{\textit{n,r}}^{tr=0}}
\DeclareMathOperator{\kk}{\Bbbk}
\DeclareMathOperator{\limi}{lim}
\DeclareMathOperator{\colim}{colim}

\DeclareMathOperator{\len}{length}
\DeclareMathOperator{\dime}{dim}
\DeclareMathOperator{\Supp}{Supp}
\DeclareMathOperator{\Hilb}{Hilb}
\DeclareMathOperator{\codime}{codim}
\DeclareMathOperator{\lendi}{End}
\DeclareMathOperator{\Hom}{Hom}
\DeclareMathOperator{\lker}{ker}
\DeclareMathOperator{\Grass}{Gr}
\DeclareMathOperator{\imag}{im}
\DeclareMathOperator{\Zmax}{Z_{\textit{n}}^{\textit{l}_{max}(\textit{r})}}
\DeclareMathOperator{\Znrl}{Z_{\textit{n}}^{\textit{l}}}
\DeclareMathOperator{\Znplurl}{Z_{\textit{n}+1}^{\textit{l}}}
\DeclareMathOperator{\Zinfrl}{Z_{\infty}^{\textit{l}}}
\DeclareMathOperator{\ZnrlwY}{\tilde{Z}_{\textit{n}}^{\textit{l}}}
\DeclareMathOperator{\ZnrlwYmax}{\tilde{Z}_{n}^{l_{max}(r)}}
\DeclareMathOperator{\quminus}{Quot_d(\mathcal{O}_{\mathbb{A}^{n-1}}^{\oplus \textit{r}})}
\DeclareMathOperator{\yuniv}{Y^{univ}_{\textit{n,r}}}
\DeclareMathOperator{\yunivplu}{Y^{univ}_{\textit{n+1,r}}}
\DeclareMathOperator{\yunivpludw}{Y^{univ}_{\textit{n+2,r}}}
\newcommand{\hooklongrightarrow}{\lhook\joinrel\longrightarrow}
\usetikzlibrary{matrix,arrows, cd, calc, bending, quotes}
\usepackage{layout}
\usepackage{hyperref}
\hypersetup{
    colorlinks=true,
    linkcolor=blue,
    filecolor=blue,      
    urlcolor=blue,
    citecolor=violet
    }

\title{Cohomology of the Quot scheme of an infinite affine space}

\date{}
\author{Paweł Pielasa}
\address{University of Cambridge, Department of Pure Mathematics and Mathematical Statistics}
\email{pp554@cam.ac.uk}

\begin{document}

\begin{abstract}
    We study the Quot scheme of points $\qu$. We exhibit and compute the cohomology of explicit loci in $\qu$, whose complement has codimension diverging to infinity as $n\rightarrow \infty$. In the case $1<r<\frac{d+1}{2}$ this loci is an irreducible component. The main ingredient in our proof are classification results on maximal-dimensional spaces of commutative matrices satisfying certain generating conditions. Our primary motivation is the study of the ind-scheme \[ \qin := \underset{n\rightarrow \infty}{\colim}\qu . \]
    Finally, we compute the cohomology (with integral coefficients) of the Quot scheme $\qudwa$, confirming, in the case $d=2$, a conjecture of Pandharipande.
\end{abstract}

\maketitle

\tableofcontents

\section{Introduction}

\subsection{The Quot scheme of points and its colimit}

Given a $\kk$-scheme $X$ of finite type and a coherent sheaf $\mathcal{E}$ on $X$ there exists a Quot scheme $\mathrm{Quot}_{\mathcal{E}/X}$ parametrizing quotients of $\mathcal{E}$. We are often interested in studying a component of the scheme above corresponding to quotients of a fixed length. A perhaps most fundamental instance of this construction is the Quot scheme of points on $\aff^n$, denoted by
\[ \qu , \] which parametrizes finite length quotients of the $r$-th direct sum the structure sheaf on $\mathbb{A}^n$. Any Quot scheme of length $d$ quotients of a locally free sheaf $\mathcal{E}$ on a smooth scheme $X$ étale locally looks like $\qu$, making the latter a natural object of study for understanding the geometry and singularities of even more general Quot schemes. In the language of modules, a $\kk$-point of $\qu$ is defined as an equivalence class of surjections
\[ \kk[x_1,\dots,x_n]^{\oplus r} \twoheadrightarrow M \]
with $M$ a length $d$ module. This suggests that the Quot scheme $\qu$ carries algebraic information about the modules it parametrizes and studying its structure can lead to a deeper outlook on the theory of modules.\\
Quot schemes have been studied extensively ever since their introduction by Grothendieck, see for example \cite{ogolnequot1}, \cite{ogolnequot2}, \cite{ogolnequot3}, \cite{ogolnequot4}, \cite{ogolnequot5}, \cite{ogolnequot6}. In the work of \cite{JS22}, their geometry has been linked more closely to the geometry of varieties of commuting matrices. This outlook will be especially useful to us, allowing us to employ tools from linear and commutative algebra to study the geometry of $\qu$.\\

Say we are given a family of moduli spaces $\mathcal{X}_n$ depending on a parameter $n$, where $n$ signifies the dimension of some ambient space, bound to some invariant etc. One of the fundamental questions to ask when trying to understand such a moduli space is: what are its properties for large enough $n$? In the case of the projective space $\mathbb{P}^n$, or the Grassmanian $\Grass(d,n)$ for $d\leq n$, we have natural embeddings from the space with parameter $n$ to the space with parameter $n+1$, that is
\[ \Grass(d,n) \overset{\iota_n}{\hookrightarrow} \Grass(d,n+1),\ \text{ or }\ \mathbb{P}^n \overset{\iota_n}{\hookrightarrow} \mathbb{P}^{n+1}.\]
The direct limits
\[ \Grass(d,\infty) := \underset{n \rightarrow \infty}{\colim}\ \Grass(d,n), \text{ respectively } \mathbb{P}^{\infty} :=  \underset{n \rightarrow \infty}{\colim}\ \mathbb{P}^{n}\]
under the sequences of embeddings above are interesting spaces in their own right, parametrising $d$-dimensional subspaces (respectively, lines going through the origin) in vector spaces of arbitrary high dimension. They both play significant roles in algebraic geometry, being the classifying spaces of rank $d$ vector bundles (respectively, line bundles). Intuitively, the colimit construction glues the schemes $X_n$ along the prescribed embeddings and eliminates the constraint of parametrising objects lying in some space of bounded dimension (necessarily sacrificing the finite dimensionality of the scheme in the process). It is often the case that the cohomology of the space $\mathcal{X}_n$ we start from agrees with the cohomology of the colimit $\underset{n\rightarrow \infty}{\colim}\ \mathcal{X}_n$ in some stable range. Then the cohomology of such an ind-scheme can be considered as the stable cohomology of the sequence of spaces $\mathcal{X}_n$. Its study is an important and fruitful area of research, enhancing the understanding of the structure of the bounded moduli spaces themselves. It is often better behaved than the cohomology of the individual schemes in the colimit, for example
\[ H^*(\Grass(d,\infty),\mathbb{Z}) \simeq \mathbb{Z}[c_1,\dots,c_d], \]
where $c_i$ is a generator in degree $2i$ for $i=1,\dots,d$.\\

There has been significant progress made in the study of the stable cohomology of the moduli of curves $\mathcal{M}_g$, ever since the conjecture of Mumford \cite{Mumford83}, stating that the rational cohomology of $\mathcal{M}_g$ stabilizes as $g\rightarrow \infty$ to
\[ \mathbb{Q}[\kappa_1,\kappa_2,\dots], \]
where the $\kappa$-class $\kappa_i$ lies in degree $2i$ for $i=1,\dots,g$. Explicitly, the class $\kappa_i$ is defined as a pushforward $\pi_*(c_1(\omega_{\mathcal{C}/\mathcal{M}_g})^{i+1})$, where $\mathcal{C}$ is the universal curve over $\mathcal{M}_g$, and $\omega_{\mathcal{C}/\mathcal{M}_g}$ is the relative canonical. It was proved in \cite{Harermodulicurves} that the cohomology indeed stabilizes, in \cite{loopspace} that it in fact stabilizes to a cohomology of some infinite loop space, and the conjecture was finally resolved in \cite{mumfordconjproof}. However, understanding the stabilization maps is slightly more cumbersome. Namely, we use that the cohomology of $\mathcal{M}_g$ agrees with the cohomology of the moduli of curves with one boundary component $\mathcal{M}_g^1$ in some stable range \cite{Harermodulicurves}. Then the maps can be defined on the level of $\mathcal{M}_g^1$ and the map between $\mathcal{M}_g^1$ and $\mathcal{M}_{g+1}^1$ is the so called Harer stability map\footnote{We thank Sam Payne for explaining this during the Workshop on Combinatorics of Enumerative Geometry}. There has also been significant progress made in understanding the cohomology of $\mathcal{M}_g$ outside of the stable range, see \cite{Payneweight2}, \cite{Payneweight11}. Similar questions have been studied for many other moduli spaces, including Hurwitz spaces \cite{Aaronhurwitz}, moduli of degree $d$ branched covers of the square \cite{squarecover}, sheaves on surfaces \cite{modulisurfaces}, etc.\\

A particularly beautiful instance of the general problem above has been studied in \cite{HJ25} when $\mathcal{X}_n$ is the Hilbert scheme of points on $\aff^n$, denoted by $\Hilb_d(\mathbb{A}^n)$. This Hilbert scheme is a special instance of the Quot scheme of points $\qu$, when $r=1$. Its $\kk$-points are equivalence classes of length $d$ algebras $A$ together with a surjection
\[ \kk[x_1,\dots,x_n] \twoheadrightarrow A . \]
The embedding
\[ \Hilb_d(\mathbb{A}^n) \overset{\iota_n}{\hooklongrightarrow} \Hilb_d(\mathbb{A}^{n+1}) \]
is given by considering a quotient $\kk[x_1,\dots,x_n] \twoheadrightarrow A$ as a quotient of $\kk[x_1,\dots,x_{n+1}]$, where the variable $x_{n+1}$ is mapped to zero, and the surjection is the same on the subring generated by the first $n$ variables. The authors of \cite{HJ25} proved that the colimit
\[ \Hilb_d(\aff^{\infty}) := \underset{n\rightarrow \infty}{\colim} \Hilb_d(\mathbb{A}^n) \]
is $\aff^1$-homotopy equivalent to the infinite Grassmanian of $(d-1)$-planes $\Grass(d-1,\infty)$, which is a condition guaranteeing an isomorphism of many algebraic invariants, including cohomology, algebraic cobordism, homotopy invariant $K$-theory etc. In particular
\[ H^*(\Hilb_d(\aff^{\infty})) \simeq \mathbb{Z}[c_1,\dots,c_{d-1}]. \]
The main ingredient of their proof is a Rees algebra construction, which exhibits an $\aff^1$-homotopy deforming any length $d$ algebra to a local algebra with square zero maximal ideal. This was used to construct an equivalence of the stack $\mathcal{F}\mathcal{F}lat_d$ of finite flat length $d$ algebras, to the stack $\mathcal{V}ect_{d-1}$ of rank $d$ locally free sheaves. Explicitly, the morphism between the stacks sends a rank $d-1$ vector bundle $\mathcal{V}$ on $X$ to the sheaf of algebras given by the square-zero extension $\mathcal{O}_X \oplus \mathcal{V}$, that is with multiplication satisfying $\mathcal{V}^2=0$ \cite[Theorem 2.1]{HJ25}.\\
It is natural to ask if we can extend those results to the case of arbitrary $r$ and find the cohomology of the colimit of the Quot schemes of points
\[ \qin := \underset{n\rightarrow\infty}{\colim}\ \qu \]
with respect to an analogous sequence of embeddings
\[ \qu \overset{\iota_n}{\hooklongrightarrow} \quplu. \]
It turns out that the approach similar to \cite{HJ25} fails, as there is no analogue of the Rees algebra in the more general setting. Nevetherless we believe that other approaches may yield the general answer. In regard to this problem, the following question was posed by Pandharipande:
\begin{Question}[{\cite[Rahul Pandharipande]{Rahul}}] \label{panda}
    Consider the colimit
    \[ \qin := \underset{n \rightarrow \infty}{\colim} \qu . \]
    Is it the case that
    \[ H^*(\qin) \simeq \mathbb{Z}[c_1,\dots,c_d]/(c_d^r) \]
    where the variable $c_i$ has degree $2i$ for $i=1,\dots,d$?
\end{Question}
The case $r=1$ holds by the aforementioned work of \cite{HJ25}.\\
In this work we use classical cohomological techniques to obtain results about components in $\qu$, whose complement has codimension converging to infinity with $n \rightarrow \infty$. In light of purity theorems in algebraic geometry and motivic homotopy theory, this makes them relevant for understanding the general answer. Furthermore, we compute the cohomology of $\qu$ and $\qin$ in the case $d=2$, confirming the special case of the conjecture above.

\subsection{Outline of our work}

In Section \ref{Ch2prelim} we outline the preliminaries to this work. We start by defining the Quot scheme of points and the morphism $\pi: \yuniv \rightarrow \qu$, which will allow us to define open, respectively closed, subsets on $\qu$ using subsets of a better understood scheme $\yuniv$ invariant under a certain group action. We introduce the resolution of singularities of the Quot scheme $\qudwa$, following \cite{S24}. We describe a version of the Białynicki-Birula decompositon of \cite{HV15}, which we use in our cohomology computations for $\qudwa$.\\

In Section \ref{sec1} we analyse the dimensions of loci in $\qu$ determined by the dimension of the linear span in $\lendi(M)$ of the endomorphisms defined by the action of the variables $x_i$ on the module $M$. We prove that one of this loci has complement of codimension diverging to infinity for $n\rightarrow \infty$.\\

In Section \ref{sec2subspaces} we study the commutative subspaces of $\lendi(V)$ for a $d$-dimensional space $V$ satisfying a certain spanning condition for a fixed number $1\leq r \leq d$. The condition arises from the surjectivity of a homomorphism $\kk[x_1,\dots,x_n]^{\oplus r} \twoheadrightarrow M$, defining a point of $\qu$. We cite the classical theorems of Schur and Jacobson in the $r \geq \frac{d}{2}$ case. In the case $1 < r < \frac{d+1}{2}$ we give new classification results, after introducing a number of technical lemmas from commutative and linear algebra. Lastly, we present a classification of subspaces of commutative endomorphisms in the case $d \leq 3$.\\

In Section \ref{sec3coh} we use the results of Section \ref{sec2subspaces} to explicitly describe certain loci in $\qu$, which by the results of Section \ref{sec1} have complement of codimension diverging to infinity for $n \rightarrow \infty$. We compute their cohomology using the Białynicki-Birula decomposition by interpreting the surjection $\kk[x_1,\dots,x_n]^{\oplus r} \twoheadrightarrow M$ as some geometric data. We hope that in the future, we will be able to use those results to compute the cohomology of $\qin$ for general $d$, at least in the case $1<r<\frac{d+1}{2}$. \\

In Section \ref{sec4quot2} we compute the Hilbert-Poincaré series of the cohomology ring of $\qudwa$. Our calculation confirms the conjecture of Pandharipande (Question \ref{panda}) in the cases $d=1$ and $d=2$. We reduce our computation from the singular Quot scheme to the smooth scheme $\Hilb_2(\aff^n \times \mathbb{P}^{r-1})$ using a resolution of singularities given in \cite{S24} and the long exact sequence in cohomology for abstract blowups. We then use a refined version of the Białynicki-Birula decomposition introduced in \cite{HV15} to find a cell decomposition of the subscheme of $\Hilb_2(\aff^n \times \mathbb{P}^{r-1})$ whose embedding induces an isomorphism on cohomology.

\subsection{Acknowledgements}

This article is based on the author's Bachelor thesis at the University of Warsaw, supervised by Joachim Jelisiejew. I would like to express my deepest gratitude to Joachim Jelisiejew for enlightening discussions on the subject and countless, invaluable remarks to help improve the previous drafts of the article.

\section{Preliminaries and notation} \label{Ch2prelim}

\subsection{Notation}
Throughout this work all schemes considered are $\kk$-schemes over a fixed algebraically closed field $\kk$ of characteristic zero. The assumption that $\kk$ is algebraically closed is used in the proof of Theorem \ref{proofkohRrn} and the study of $\qu$ when $d = 2$, more specifically for the classification of length $2$ modules and applying the Białynicki-Birula decomposition to $\qudwa$ in Section \ref{sec4quot2}.
Let $S$ be the polynomial ring $\kk[x_1,\dots,x_n]$ in $n$ variables. We denote by $S^{\oplus r}$ the rank $r$ free module on $S$, and by $e_1,\dots,e_r$ its canonical basis. We denote by $\Grass(a,b)$ for $a \geq b$ the Grassmanian of $b$-dimensional quotients of an $a$-dimensional vector space. We denote by $\Grass(a,\infty)$ the infinite Grassmanian of $a$-planes. Unless stated otherwise, the cohomology $H^*(X)$ is assumed to be with integral coefficients.

\subsection{Quot schemes of points}

In this section we recall the notions fundamental to the study of the Quot Scheme $\qu$.


\begin{Definition}
    Let $X$ be a quasi-projective $\kk$-scheme and let $\mathcal{E}$ be a quasi-coherent sheaf on $X$. We denote by $\mathrm{Quot}_{\mathcal{E}/X}$ the functor $Sch_{\kk} \rightarrow Set$ parametrizing quotients of $\mathcal{E}$. Explicitly, for a $\kk$-scheme $T$:
        \[
    \mathrm{Quot}_{\mathcal{E}/X}(T) = \left\{ 
    \mathcal{E}_T \twoheadrightarrow \mathcal{F} \;\middle|\;
    \begin{array}{l}
    \mathcal{F} \in \mathrm{QCoh}(X_T),\ \mathcal{F} \text{ is of finite presentation,} \\
    \mathcal{F} \text{ is flat over } T,\ \mathrm{\Supp}(\mathcal{F}) \text{ is proper over } T
    \end{array}
    \right\} \biggl/ \sim
    \]
    where $\sim$ denotes the obvious equivalence relation.
\end{Definition}

    Quasi-coherent sheaves over an affine scheme $\mathrm{Spec}(R)$ are in a one to one correspondence with modules over the ring $R$, with the correspondence being given by the global sections functor.
Our main object of study is the subscheme of the Quot scheme of $\mathcal{E} = \mathcal{O}_{\mathbb{A}^n}^{\oplus r}$ on $X = \mathbb{A}^n$  consisting of quotients of fixed length $d$.
\begin{Definition}
    We denote by $\qu$ the subscheme of $\mathrm{Quot}_{\mathcal{O}_{\mathbb{A}^n}^{\oplus r}/\mathbb{A}^n}$ parametrizing quotients of length $d$. Explicitly, for a $\kk$-algebra $R$:
    \begin{align*}
        \qu(R) = \Biggl\{ S_R^{\oplus r} \twoheadrightarrow M\ \bigg|
        \begin{array}{ccc}
            M \in Mod_{R[x_1,\dots,x_n]},\ S^{\oplus r}_R \twoheadrightarrow M \text{ is a surjection,} \\ 
            M \text{ is a locally free }R \text{-module of rank }d
        \end{array}
        \Biggr\} \biggl/ \sim
    \end{align*}
    where $\sim$ denotes the obvious equivalence relation. 
\end{Definition}
The $\kk$-points of $\qu$ correspond to the quotients of $S^{\oplus r}$ that have dimension $d$ as a $\kk$-vector space.

Fix a $d$-dimensional $\kk$-vector space $V$. Understanding the topology of $\qu$ using Grothendieck's construction can be somewhat cumbersome. We instead use a morphism from another scheme, which we denote by $\yuniv$ and which has a straightforward description as a subscheme of $(\lendi (V))^{\times n} \times V^{\times r}$.

\begin{Definition}
    Fix a $d$-dimensional $\kk$-vector space $V$. We define $\yuniv$ to be the subscheme of $\lendi (V)^{\times n} \times V^{\times r}$ given by tuples $((X_1,\dots,X_n),(v_1,\dots,v_r))$ such that $X_i$ are pairwise commutative and $k[X_1,\dots,X_n] U=V$ for $U = \langle v_1,\dots,v_r \rangle$.
\end{Definition}

We can consider projections $pr_1: \yuniv \rightarrow \lendi (V)^{\times n}$ and $pr_2: \yuniv \rightarrow V^{\times r}$ induced by the respective projections from $\lendi (V)^{\times n} \times V^{\times r}$. There is an action of $GL(V)$ on $(\lendi (V))^{\times n} \times V^{\times r}$ given by
\begin{align*}
    \big(g,\big((X_1,\dots,X_n),(v_1,\dots,v_r)\big)\big) \mapsto \big((gX_1g^{-1},\dots,gX_ng^{-1}),(gv_1,\dots,gv_r)\big), 
\end{align*}
which preserves $\yuniv$. Any $\kk$-point of $\yuniv$ gives rise to a $\kk$-point of $\qu$ defined as the unique quotient $S^{\oplus r} \twoheadrightarrow V$ of $S$-modules sending $e_j \mapsto v_j$ for $j=1,\dots,r$ where the variables $x_j$ of $S$ act on $V$ by $X_j$. The commutativity condition ensures that this is gives a structure of an $S$-module on $V$ and the spanning condition ensures that the homomorphism is a surjection. In fact, this construction defines a morphism of schemes that is very well behaved. More concretely:
\begin{Proposition}[{\cite[Proposition 3.7]{JS22}}] \label{defpi}
    There is a smooth, surjective morphism
    \[ \pi: \yuniv \rightarrow \qu \]
    defined on $\kk$-points as above.
\end{Proposition}
In the other direction, if for a quotient $S^{\oplus r} \twoheadrightarrow M$ we fix an isomorphism $M \simeq V$ of $\kk$-vector spaces, we can define a point of $\yuniv$ by the inverse of the above construction.
\begin{Proposition}[{\cite[Lemma A.2(ii)]{BaranovskyADHM}}]
    The points in the fiber of $\pi$ over $[M] \in \qu$ differ by a change of basis in $V$. More precisely, let $y$ be any element such that $\pi(y) = [M]$. Then $\pi^{-1}([M])$ is exactly the orbit of $y$ under the $GL(V)$-action on $\yuniv$.
\end{Proposition}

Since smooth morphisms are open \cite[Tag01V4]{stacks-project}, $\pi$ can be used to define open sets in $\qu$ as images of open sets in $\yuniv$. To define closed subsets of $\qu$ using $\pi$, we introduce the following.
\begin{Definition} \label{definvariant}
    We call a subset $Z$ of $\yuniv$ \textit{invariant} if its intersection with every fiber of $\pi$ is either empty or the whole fiber. Equivalently, $Z$ is preserved by the $GL(V)$ action on $\yuniv$.
\end{Definition}
The condition above holds for the subsets of $\qu$ defined invariantly with respect to the choice of basis of $V$, for example by the dimensions of $\lker(X_i)$, $\langle X_1, \dots, X_n \rangle \subset \lendi (V)$, etc. For conciseness we often refer to such conditions as \textit{invariant}, especially when considering them to define loci of $\qu$ with the use of the following proposition.
\begin{Proposition} \label{pipreservesclosed}
    Let $Z$ be a closed invariant subset of $\yuniv$. Then its image under $\pi$ is closed.
\end{Proposition}
\begin{proof}
    The image of the complement of an invariant subset is equal to the complement of its image. Thus the complement of the image of $Z$ is open, so the image of $Z$ is closed.
\end{proof}
 Combining the above proposition with the previous discussion, we see that the image of a locally-closed invariant subset of $\yuniv$ is locally closed in $\qu$.\\

Understanding the linear spans $\langle pr_1(y) \rangle \subset \lendi(V)$ for $y \in \yuniv$ is a key component of this work. To make the proofs more concise, we introduce the following two notions defining properties of a subspace of $\lendi(V)$, which are satisfied for all points of the above form.
\begin{Definition}\label{defcommute}
    We say that a subspace $A\subset \lendi(V)$ is commutative if it consists of pairwise commutative endomorphisms.
\end{Definition}
\begin{Definition}\label{defrspanning}
    We say that a subspace $A\subset \lendi(V)$ is $r$-spanning if there exists a subspace $U \subset V$ od dimension at most $r$ such that $A\cdot U = V$. We sometimes refer to this as the $r$-spanning condition, or simply the spanning condition when the value of $r$ is clear from the context.
\end{Definition}

Consider a sequence of embeddings of affine spaces
\begin{align*}
    \ldots \overset{\iota_{n-1}}{\hooklongrightarrow} \mathbb{A}^{n} \overset{\iota_{n}}{\hooklongrightarrow} \mathbb{A}^{n+1} \overset{\iota_{n+1}}{\hooklongrightarrow} \mathbb{A}^{n+2} \overset{\iota_{n+2}}{\hooklongrightarrow} \dots
\end{align*}
induced by the surjections of polynomial rings $\kk[x_1,\dots,x_{n+1}] \twoheadrightarrow \kk[x_1,\dots,x_n]$ sending $x_j \mapsto x_j$ for $j\leq n$ and $x_{n+1} \mapsto 0$. The main motivation for our considerations is studying the colimit of embeddings of Quot schemes under the induced maps
\[ (\iota_n)_*: \qu \hookrightarrow \quplu. \]
More precisely, since the pushforward with respect to a closed embedding is an exact functor, any closed immersion $f: X \rightarrow Y$ induces a morphism of Quot schemes
\begin{align*}
    \mathrm{Quot}_{\mathcal{E}/X}(S) \overset{f_*}{\longrightarrow} \mathrm{Quot}_{f_*\mathcal{E}/Y}(S).
\end{align*}
Since the pushforward with respect to a closed immersion is a conservative functor the induced morphism between Quot schemes is an embedding.
In our case the embeddings $\iota_n$ induce the embeddings
\begin{align} \label{embeddings}
    \ldots \overset{(\iota_{n-1})_*}{\hooklongrightarrow} \qu \overset{(\iota_{n})_*}{\hooklongrightarrow} \quplu \overset{(\iota_{n+1})_*}{\hooklongrightarrow} \qupludw \overset{(\iota_{n+2})_*}{\hooklongrightarrow} \dots .
\end{align}
To make the notation more transparent, in the subsequent considerations we denote the induced embeddings between the Quot schemes also by $\iota_n$. The embedding can be described explicitly as sending an element $\kk[x_1,\dots,x_n]^{\oplus r} \twoheadrightarrow M$ of $\qu$ to the quotient
\[ \kk[x_1,\dots,x_{n+1}]^{\oplus r} \twoheadrightarrow \kk[x_1,\dots,x_{n+1}]^{\oplus r}/(x_{n+1}) \simeq \kk[x_1,\dots,x_n]^{\oplus r} \twoheadrightarrow M . \]

One of the main goals of this article is to study the ind-scheme given by taking the colimit of $\qu$ under the sequence of embeddings $\iota_n$.
\begin{Definition}
    We denote by $\qin$ the filtered colimit $\underset{n \rightarrow \infty
    }{\colim} \qu$ with the embeddings defined by $\iota_n$.
\end{Definition}
The study of the scheme $\qin$ motivates understanding the properties of the cohomology of Quot scheme $\qu$ for $n\gg 0$, or preserved under taking the colimit above.\\

The colimit can be also understood as induced by an analogous sequence of embeddings
\begin{align*}
    \ldots \overset{\iota_{n-1}}{\hookrightarrow} \yuniv \overset{\iota_{n}}{\hookrightarrow} \yunivplu \overset{\iota_{n+1}}{\hookrightarrow} \yunivpludw \overset{\iota_{n+2}}{\hookrightarrow} \dots
\end{align*}
where $\iota_n$ is defined by mapping a tuple of elements to one extended by $X_{n+1} = 0$, that is
\begin{align*}
    ((X_1,\dots,X_n),(v_1,\dots,v_r)) \mapsto ((X_1,\dots,X_n,0),(v_1,\dots,v_r)).
\end{align*}

\subsection{\texorpdfstring{A resolution of $\qudwa$}{A resolution of qudwa}}

In this subsection recall the results of \cite{S24}, giving a resolution of singularities for the Quot scheme of points $\mathrm{Quot}_2(\mathcal{O}_{\mathbb{A}^n}^{\oplus r})$ on an $n$-dimensional affine space. In the original setting of the paper, the author considers the Quot scheme $\mathrm{Quot}^2_{\mathcal{E}/\mathcal{S}}$, which parametrises length-two quotients of a locally free sheaf on a surface $\mathcal{S}$. However, the arguments apply equally well to the Quot scheme parametrising quotients of a locally free sheaf on an $n$-dimensional affine space $\mathbb{A}^n$. We use the resolution to reduce the computation of the cohomology of the Quot scheme to the computation of the cohomology of its smooth resolution of singularieties. There we are able to apply the Białynicki-Birula decomposition in our computations.\\


Let us first introduce another scheme, which serves as the resolution of singularities. Let $X$ be a scheme of finite type and let $d\in \mathbb{N}$. Then we define
\begin{align*}
    \Hilb_d(X)(T) := \Biggl\{ Z \subset X_T \ \bigg|
    \begin{array}{c}
        Z \subset X_T\ \text{is a closed subscheme such that }\\ Z\rightarrow T \text{ is finite locally free of degree }d
    \end{array} \Biggr\}.
\end{align*}
This is in fact a special case of the Quot scheme of finite length quotients $\mathrm{Quot}^d_{\mathcal{E}/X}$ in the case $\mathcal{E}$ is the structure sheaf $\mathcal{O}_X$. \\

Let $\mathcal{E}$ be a locally free sheaf on $\mathbb{A}^n$. Define $\mathcal{O}_{\mathbb{P}\mathcal{E}}(1)$ to be a canonical line bundle on the projectivization $\mathbb{P}\mathcal{E}$ and let $p: \mathbb{P}\mathcal{E} \rightarrow \mathbb{A}^n$ be the projection map. Denote by $\mathrm{Tot}(\mathbb{P}(\mathcal{E}))$ the total space of the projectivization of $\mathcal{E}$. We can now state a result that is a straightforward generalization of \cite[Theorem 1]{S24}, extending it from the surface case to the affine space $\mathbb{A}^n$.

\begin{Theorem} \label{resolution}
    Fixing the notation as above, there exists a morphism of schemes
    \begin{align*}
        \rho^2: \mathrm{Hilb}_2(\mathrm{Tot}(\mathbb{P}(\mathcal{E}))) \rightarrow \mathrm{Quot}_{\mathbb{A}^n}^{2}(\mathcal{E})
    \end{align*}
    induced by mapping a length two subscheme $Z$ of $\mathrm{Tot}(\mathbb{P}(\mathcal{E}))$ to the quotient
    \begin{align*}
        \mathcal{E} = p_*\mathcal{O}_{\mathbb{P}(\mathcal{E})}(1) \twoheadrightarrow p_*\mathcal{O}_Z(1).
    \end{align*}
    This morphism is a resolution of singularities of $\mathrm{Quot}_{\mathbb{A}^n}^{2}(\mathcal{E})$.
\end{Theorem}

In the case we are considering there is an isomorphism $\mathcal{E} \simeq \mathcal{O}^{\oplus r}_{\mathbb{A}^n}$, and hence $\mathrm{Tot}(\mathbb{P}(\mathcal{E})) \simeq \mathbb{A}^n \times \mathbb{P}^{r-1}$. This is a quasi-projective scheme whose topology we can understand very well. We use this to our advantage in Section \ref{sec4quot2} to analyse the cohomology of the Hilbert scheme $\Hilb_2(\mathrm{Tot}(\mathbb{P}(\mathcal{E})))$.\\

To be able to use the resolution of singularities above we need to know what is the singular locus of $\qudwa$.
\begin{Theorem}[{\cite[Corollary 3.1]{S24}}] \label{gladkilocusqu}
    Let $\mathcal{Q}$ be the universal quotient sheaf on $\qudwdw$. Then a point $[S^{\oplus r} \twoheadrightarrow M]\in \qudwdw$ is smooth if and only if $h^0\big(\mathbb{A}^2,\mathcal{E}nd(\mathcal{Q}_{M})\big) = 2$. In the language of modules, this is equivalent to $\dime(End_{\kk[x_1,x_2]}(M)) = 2$.
\end{Theorem}
Using the classification of commutative subspaces of $2 \times 2$ matrices in Section \ref{sec2subspaces}, the condition can be stated equivalently as follows: $[S^{\oplus r} \twoheadrightarrow M]$ is a smooth point if and only if the subalgebra generated by $S$ in $\lendi (M)$ is not equal to $\langle Id \rangle$. Since $\textrm{Quot}_2(\mathcal{O}_{\aff^2}^{\oplus r})$ is irreducible \cite{JS22}, and of dimension $2r+2$, the dimension of the tangent space at a point $[S^{\oplus r} \twoheadrightarrow M]\in \textrm{Quot}_2(\mathcal{O}_{\aff^2}^{\oplus r})$ is equal to $2r+2$ if and only if the point is smooth.
\begin{Theorem}[{\cite[Theorem 3.1]{S24}}] \label{wymiarstycznejquotpkt}
    For every point $[M] \in \quddw$,
    \[ \dime\mathcal{T}_{[M]} = d\cdot r+h^0(\aff^2,\mathcal{E}nd(\mathcal{Q}_{[M]})) \]
    where $\mathcal{Q}$ is the universal quotient sheaf on $\quddw$.
\end{Theorem}
Let us take note of the following classical fact about Quot schemes, which we use in Section \ref{sec4quot2} when extending Theorem \ref{gladkilocusqu} to $\qudwa$.
\begin{Proposition}[{\cite[Theorem 2.7.3]{hartshorne2009deformation}}] \label{stycznadoquot}
    The tangent space to the Quot scheme $\qu$ at a point $[S^{\oplus r} \twoheadrightarrow M]$ is given by $\Hom_S(K,M)$, where $K$ is the kernel of $S^{\oplus r} \twoheadrightarrow M$.
\end{Proposition}

\subsection{Long exact sequence of an abstract blowup}

We now introduce a long exact sequence, which allows us to compute the cohomology of a possibly singular scheme using the cohomology of its resolution of singularities. We apply it in Section \ref{sec4quot2} to compute the cohomology of a the Quot scheme $\qudwa$, which is singular. First we define the notion of an abstract blowup, to which the theorem applies.

\begin{Definition}
    Let $X$ be a scheme of finite type over the field $\kk$ and let $Z \subset X$ be a closed immersion. A proper morphism $Y \rightarrow X$ is an \textit{abstract blowup} with center $Z$ if there exists an isomorphism of the reduced loci $(X-Z)_{\mathrm{red}} \simeq (Y-Z')_{\mathrm{red}}$, where $Z' = Z \times_X Y$.
\end{Definition}

Given an abstract blowup $Y \rightarrow X$, we can write an exact sequence, which will be immensely useful in our computations.

\begin{Theorem}[{\cite[14.5.3]{MVW06}}] \label{blowupseq}
    Let $Y \rightarrow X$ be an abstract blowup with center $Z$ and set $Z' = Z \times_X Y$. Then there exists a long exact sequence in cohomology
    \begin{align*}
        \ldots \rightarrow H^i(X) \rightarrow H^i(Y) \oplus H^i(Z) \rightarrow H^i(Z') \rightarrow H^{i+1}(X) \rightarrow \dots
    \end{align*}
\end{Theorem}

The sequence in singular cohomology above is obtained from the sequence of motives in \cite{MVW06} by applying the contravariant functor $\Hom(-,\mathbb{Z})$. It generalizes in many ways, including to étale cohomology, Chow groups.

\subsection{The Białynicki-Birula decomposition}

In this subsection we introduce a version of the Białynicki-Birula decomposition \cite{ABB76}, which is a powerful tool for computations of cohomology. It has been successfully applied to research the cohomology and geometry of moduli spaces, notably the Hilbert scheme of points, see \cite{G90}, \cite{JJ19} and moduli of Higgs bundles, see \cite{HV15}, \cite{H22}. We present a result giving a decomposition under a weakened assumption of \textit{semiprojectivity}, following \cite{HV15}. For a detailed description of the classical version see for example \cite{BCM02}.\\

\begin{Definition}[{\cite[Def 1.1.1]{HV15}}] \label{defsemiprojective}
    Let $X$ be a smooth, quasi-projective variety over $\mathbb{C}$ with an action of the torus $\mu: \Gm \times X \rightarrow X$. Then we say that $X$ is \textit{semiprojective} if the following conditions hold
    \begin{enumerate}
        \item The fixed point set $X^{\Gm}$ is proper.
        \item For any point $x\in X$ there exists a limit $\underset{t\rightarrow 0}{\limi}\ \mu(t,x)$.
    \end{enumerate}
\end{Definition}
To make the notation more transparent, whenever a torus acts on a variety $X$ we write $t\cdot x := \mu(t,x)$.\\
Let $X^{\Gm} = \coprod_{i\in I} Z_i$ be the decomposition of the fixed locus into irreducible components. Define the \textit{positive cell} $X_i^+$ to be the locus of closed points with a limit at $t\rightarrow 0$ in the component $Z_i$, that is
\begin{align*}
X_i^+ := \{x\in X \mid \underset{t\rightarrow 0}{\limi}\ t \cdot x \in Z_i \}. 
\end{align*}
By definition for a semiprojective variety $X$ we have a decomposition $X = \coprod_{i \in I} X_i^+$ as sets. We analogously define the \textit{negative cells} $X_i^-$ by
\begin{align*}
    X_i^- := \{x\in X \mid \underset{t\rightarrow \infty}{\limi}\ t \cdot x \in Z_i \}.
\end{align*}
\begin{Theorem}[{\cite[Białynicki-Birula theorem]{HV15}}] \label{ABBdointro}
    For a semiprojective variety $X$ the loci $X_i^+$ (respectively $X_i^-$) defined above are locally closed and form affine bundles over $Z_i$ with the morphisms $p_i: X_i^+ \rightarrow Z_i$ given by $x \mapsto \underset{t\rightarrow 0}{\limi}\ t\cdot x$ (respectively $\underset{t\rightarrow \infty}{\limi} t\cdot x$).
\end{Theorem}
In the classical, projective case this gives a decomposition of $X$ into locally closed cells $X_i^+$ (respectively $X_i^-$) \cite{BCM02}, which can often be much better understood than $X$ itself. If $X$ is semiprojective, the same still holds for the positive cells $X_i^+$.\\
A case of particular interest is when $X^{\Gm}$ consists of a finite number of points. Then $X_i^+$ (respectively $X_i^-$) form affine bundles over points, and thus they are affine spaces. This can be used to extract the cohomology of the projective variety $X$ from its cell decompositon. Describing the action of $\Gm$ on the tangent space to a point $Z_i$ of the fixed locus allows us to compute the dimensions of the cells of the decomposition. More precisely:
\begin{Lemma}[{\cite[Theorem 4.2.(4)]{BCM02}}] \label{styczne}
    Let $X$ be a semiprojective variety such that $X^{\Gm}$ consists of finitely many points. Consider the restriction $(\mathcal{T}_X)_{|Z_i}$ of the tangent bundle of $X$ to a point $Z_i$. Let
    \[ (\mathcal{T}_X)_{|Z_i} = \mathcal{T}_i^+ \oplus \mathcal{T}_i^- \]
    be a decomposition of $(\mathcal{T}_X)_{|Z_i}$ into a direct sum of subbundles where $\Gm$ acts with a positive or negative character. Then there exists an isomorphism of affine spaces $X_i^+ \simeq \mathcal{T}_i^+$ (respectively $X_i^- \simeq \mathcal{T}_i^-$) compatible with $p_i$.
\end{Lemma}
For a semiprojective variety $X$ not all points have to have limits when $t\rightarrow \infty$. This motivates the following definition.
\begin{Definition}
    Let $X$ be a semiprojective variety. Then the sum of the negative cells $\mathcal{C} := \bigcup_{i\in I} X_i^-$ is the \textit{core} of $X$.
\end{Definition}
The importance of understanding the core of $X$ lies in the following two theorems.
\begin{Theorem}[{\cite[Corollary 1.2.2]{HV15}}] \label{coreproper}
    The core $\mathcal{C}$ of a semiprojective variety $X$ is proper.
\end{Theorem}
\begin{Theorem}[{\cite[Theorem 1.3.1]{HV15}}] \label{coreizo}
    The inclusion of the core $\iota: \mathcal{C} \hookrightarrow X$ induces an isomorphism $\iota^*: H^*(X,\mathbb{Z}) \overset{\sim}{\longrightarrow} H^*(\mathcal{C},\mathbb{Z})$.
\end{Theorem}
The reason why we care about the properness of the core $\mathcal{C}$ is that in our setting it implies projectivity.
\begin{Lemma}[{\cite[Tag 0B44]{stacks-project}}]
    If a scheme $X$ over $\kk$ is quasiprojective and proper, then it is projective.
\end{Lemma}
This allows us to use the classical Białynicki-Birula decomposition for projective varieties.

\section{\texorpdfstring{Loci of $\qu$ and their dimensions}{Loci of qu and their dimensions}} \label{sec1}

In this section we introduce a decomposition of the Quot scheme $\qu$ into particular locally closed loci. We compute the order of growth with $n$ of their dimensions and use this to exhibit a single locus, whose complement has codimension converging to infinity. This is naturally motivated by purity theorems, for example the Morel-Voevodsky purity isomorphism \cite{MV99}, whose important corollary is the following:
\begin{Theorem}
    Let $Z \hookrightarrow X$ be a closed immersion of smooth connected $\kk$-schemes such that $\codime_X(Z) \geq r$. Then, letting $U = X \backslash Z$ be the complement of $Z$, the inclusion $\iota: U \hookrightarrow X$ induces an isomorphism in cohomology in degrees $< r$, that is $\iota^*: H^i(U) \simeq H^i(X)$ for $i<r$.
\end{Theorem}
 Consider a sequence of smooth, closed subschemes $Z_i \subseteq X_i$ with $X_i$ smooth for $i=1,2,\dots$ and a sequence of embeddings $X_i \hookrightarrow X_{i+1}$ compatible with $Z_i$. Assume that $\underset{i\rightarrow 
\infty}{\limi} \codime_{X_i}(Z_i) = \infty$. Then since cohomology commutes with filtered colimits, letting $Z = \underset{i\rightarrow 
\infty}{\colim} Z_i$ and $X = \underset{i\rightarrow 
\infty}{\colim} X_i$, the theorem above implies there is an isomorphism $H^*(X) \simeq H^*(X\backslash Z)$. While this does not apply in the setting of the non-smooth Quot scheme, it is our hope that a good understanding of its high dimensional loci will lead to further results on its cohomology and $\mathbb{A}^1$-homotopy.\\

The morphism $\pi: \yuniv \rightarrow \qu$ (see Proposition \ref{defpi}) allows us to define locally closed subsets of $\qu$ via invariant subsets of $\yuniv$, defined as in Definition \ref{definvariant}. For each point $y \in \yuniv$ we can consider the projection $pr_1(y) \in \lendi (V)^n$. By definition of $\yuniv$ this is an $n$-tuple of pairwise commutative endomorphisms in $\lendi (V)$ and we can consider the linear span of those $n$ elements in $\lendi (V)$. It is a subspace of commutative matrices. We use a shorthand notation $\langle pr_1(y) \rangle$ to refer to this linear subspace. Since the subsets defined by fixing the dimension of $\langle pr_1(y) \rangle$ are invariant under the change of basis of $V$, by Proposition \ref{pipreservesclosed} we conclude that they are locally closed in $\qu$. This motivates the following definition.
\begin{Definition}
    Define $\Znrl$ to be the locally closed subset of $x\in \qu$ defined by the invariant condition $\dime_{\kk}\langle pr_1(y) \rangle = l$ for all $y\in \pi^{-1}(x)$.
\end{Definition}
Since $\langle pr_1(y) \rangle \subset \lendi (V)$, there is a disjoint union
\[ \qu = \bigsqcup\limits_{l=1}^{d^2} \Znrl . \]
Note that the subset $\Znrl$ is mapped into $\Znplurl$ under the map
\[ \iota_n: \qu \hookrightarrow \quplu. \]

We would like to bound the dimensions of $\Znrl$.
\begin{Lemma}
    Assume that $n \geq d^2$. Let
    \[ \ZnrlwY := \pi^{-1}(\Znrl) \subset \yuniv. \]
    Then
    \[ \dime(\ZnrlwY) \leq \frac{d^4}{4} + n \cdot l + r \cdot d. \]
\end{Lemma}
\begin{proof}
    For every $y \in \ZnrlwY$ the span $\langle pr_1(y) \rangle$ defines a point in the Grassmanian $\Grass(l,\lendi (V)) = \Grass(l,d^2)$. Note that
    \[ \dime(\Grass(l,d^2)) = l \cdot (d^2-l) \leq \frac{d^4}{4}. \]
    For a fixed $l$-dimensional subspace $A \in \Grass(l,\lendi (V))$ a point in $\ZnrlwY$ with $\langle pr_1(y) \rangle = A$ is uniquely determined by a choice of $n$ spanning elements $X_i$ of $A$ and certain $r$ elements of $V$. Thus
    \begin{align*}
        \dime(\ZnrlwY) \leq \dime(\Grass(l,\lendi (V)) + n\cdot \dime(A) + r \cdot \dime(V) \leq  \frac{d^4}{4} + n \cdot l + r \cdot d.  \qedhere
    \end{align*}
\end{proof}
We would now like to find a lower bound for the dimension of $\ZnrlwY$.
\begin{Lemma}
    Assume $n \geq d^2$. Pick $l$ such that $\ZnrlwY \neq \emptyset$. Then 
    \[ \dime(\ZnrlwY) \geq n \cdot l + r \cdot d . \]
\end{Lemma}
\begin{proof}
    Fix an $l$-dimensional, $r$-spanning, commutative subspace $A$ of $\lendi(V)$. Such $A$ exists, since any linear span $\langle pr_1(y) \rangle$ for $y\in \ZnrlwY$ will suffice. Since $n\geq l$, a general choice of $n$ elements of $A$ determines the set of $X_1,\dots, X_n$ such that
    \[ \langle X_1,\dots , X_n \rangle = A. \]
    Since the condition of a tuple of vectors $(v_1,\dots,v_r) \in V^{\times r}$ satisfying $k[A] \cdot \langle v_1,\dots,v_r \rangle = V$ is open in $V^{\times r}$ and given by non-vanishing of certain minors, a general choice of $(v_1,\dots,v_r)$ satisfies it. Thus a general choice $\big((X_1,\dots,X_n),(v_1,\dots,v_r)\big)$ in $A^{\times n} \times V^{\times r}$ as above defines a point of $\ZnrlwY$, so
    \[ \dime(\ZnrlwY) \geq n \cdot \dime(A) + r \cdot d = n \cdot l + r \cdot d . \qedhere \]
\end{proof}
For a point $x \in \qu$, all points in the fiber of $\pi$ over $x$ differ by the action of $GL(V)$ on $\yuniv$, and two points differing by the $GL(V)$ action map to the same point of $\qu$ under $\pi$. It follows from the spanning condition, that the action of $GL(V)$ is free. Since $\pi$ is surjective, this implies that
\[ \dime(\pi^{-1}(x)) = \dime(GL(V)) = d^2. \]
In particular
\[ \dime(\ZnrlwY) = \dime(\Znrl) + d^2 . \]
Combining this observation with the lemmas above leads to the following theorem.
\begin{Theorem}
    Let $l \in \{0,\dots,d^2\}$, $n \geq d^2$ be such that $\Znrl \neq \emptyset$. Then
    \[ \dime(\Znrl) = n \cdot l + r \cdot d - d^2 + C(n,r,d), \]
    where
    \[ 0 \leq C(n,r,d) \leq \frac{d^4}{4}. \]
\end{Theorem}
Thus the function $\dime\big(\Znrl\big)$ grows linearly in $n$, and the linear term depends only on the value of $l$. Define $l_{\max}(r)$ to be the largest integer in $\{0,\dots,d^2\}$ such that $\Zmax \neq \emptyset$. Then
\[ \qu \backslash \Zmax = \bigsqcup\limits_{l=1}^{l_{\max}(r)-1} \Znrl \]
and thus
\begin{align*}
    \dime(\qu \backslash \Zmax) \leq \underset{l<l_{\max}(r)}{max} \dime(\Znrl) \leq n \cdot (l_{\max}(r)-1) + \left(r \cdot d -d^2+\frac{d^4}{4}\right)
\end{align*}
This is of lower order in $n$ than
\[ \dime(\qu) \geq \dime(\Zmax) = n \cdot l_{\max}(r) + r\cdot d - d^2 + C(n,r,d), \]
therefore
\[ \underset{n\rightarrow\infty}{\limi}\codime\left(\qu \backslash \Zmax\right) = \infty. \]

We define loci in $\qu$ and $\yuniv$ corresponding to surjections where $X_j$ act via traceless operators. We show that the Quot scheme (respectively $\yuniv$) is isomorphic to the product of $\mathbb{A}^n$ with this traceless locus.
\begin{Definition} \label{defQnrd}
    We denote by $\Yhat$ the loci in $\yuniv$ defined on $\kk$-points by \[ \{ y\in \yuniv \mid pr_1(y) \text{ is a tuple of traceless matrices}\} . \]
    Let $\Qnrd:=\pi(\Yhat)$.
\end{Definition}
\begin{Lemma} \label{rozkladbezsladowe}
     With the definitions above $\yuniv \simeq \mathbb{A}^n \times \Yhat$ and
\begin{align}
    \qu \simeq \mathbb{A}^n \times \Qnrd.
\end{align}
\end{Lemma}
\begin{proof}
    We first give an isomorphism for $\yuniv$ by viewing it as a subspace of $\lendi(V)^{\times r} \times V^{\times r}$, and by defining two morphisms between $\lendi(V)^{\times r} \times V^{\times r}$ and $\aff^n \times \lendi(V)^{\times r} \times V^{\times r}$. These morphisms are mutually inverse when restricted to $\yuniv$ and $\aff^n \times \Yhat$, respectively. In one direction the morphism is given by sending an $n$-tuple $(X_1,\dots,X_n)$ of matrices to
\begin{align*}
     \left(X_1-\frac{tr(X_1)}{d}Id,\dots,X_n-\frac{tr(X_n)}{d}Id\right) \in \lendi (V)^{\times n}
\end{align*} and a point
\begin{align*}
    \left(\frac{tr(X_1)}{d},\dots,\frac{tr(X_n)}{d}\right) \in \mathbb{A}^n,
\end{align*}
while acting on $V^{\times r}$ with identity.
The inverse morphism is given by sending a pair
\begin{align*}
    (\lambda_1,\dots,\lambda_n)\in \aff^n,(X_1,\dots,X_n) \in \lendi(V)^{\times n}\ \text{ to }\ (X_1+\lambda_1 Id,\dots,X_n+\lambda_nX_n),
\end{align*}
and acting on $V^{\times r}$ with identity. Since adding multiples of $Id$ to $X_i$ does not change the subalgebra they generate, the maps  restrict to isomorphisms. Since the trace and $Id$ are invariant under the change of basis, the isomorphism descends to an isomorphism on $\qu$.
\end{proof}
We use this in Section \ref{sec3coh} to replace loci of $\qu$ by products of $\mathbb{A}^n$ with loci of $\Qnrd$.\\

We observe that the constructions above are compatible with the morphisms $\iota_n$ described in the diagram \ref{embeddings}. This is because the traceless condition, as well as the dimension of the $\langle pr_1 (y) \rangle \subset \lendi (V)$, are preserved under extending the tuple $(X_1,\dots,X_n)$ of endomorphisms by $X_{n+1} = 0$. Thus we define
\begin{align*}
    \Zinfrl := \underset{n \rightarrow \infty}{\colim}\ \Znrl\ \text{ and }\  \Qnrdin := \underset{n \rightarrow \infty}{\colim}\ \Qnrd.
\end{align*}
This allows us to consider their colimits contained in the colimit of the Quot schemes under the sequence of embeddings \[ \qin = \underset{n\rightarrow\infty}{\colim}\ \qu . \] The same property holds for other loci we consider, including the ones containing $\Zmax$, whose cohomology we study in Section \ref{sec3coh}. This could make them natural candidates for the cohomology of $\qin$, assuming that we are able to prove a purity-type theorem for the non-smooth Quot scheme.\\

To better understand the topology of the Quot scheme we would like to determine what are the possible subspaces $A = \langle pr_1(y) \rangle \in \Grass(l_{\max},\lendi (V))$ for some $y\in \ZnrlwYmax$. Such a subspace $A$ has to be $l_{\max}(r)$-dimensional, commutative (Definition \ref{defcommute}), and the subalgebra of $\lendi(V)$ generated by $A$ has to be $r$-spanning (Definition \ref{defrspanning}). The $r$-spanning property follows since we can take $\langle v_1,\dots,v_r \rangle$ as the subspace witnessing the $r$-spanning condition. In Section \ref{sec2subspaces} we show that we can assume that $A$ is a subalgebra of $\lendi(V)$, which implies that the $r$-spanning condition has to be satisfied by $A$. We also determine $l_{\max}(r)$ for all possible values of $r$ and classify the subspaces $A$ above. This gives us an explicit description of the locus $\mathrm{Z}_n^{l_{\max}}$, which we use in the cohomology computations of Section \ref{sec3coh}.

\section{Maximal-dimensional commutative subspaces} \label{sec2subspaces}

Commutative subspaces of matrices have been long studied in algebra and algebraic geometry, including in regards to the study of Quot and Hilbert schemes, see, \cite{BaranovskyADHM}, \cite{BARANOVSKY01}, \cite{Gerstenhaber61commuting}, \cite{KOSIR1993}, \cite{noteadhm}, \cite{JS22}. In this section we determine exactly which subspaces of $\lendi (V)$ are of the form $\langle pr_1(y) \rangle$ for $y \in \ZnrlwYmax \subset \yuniv$. We use this later to define certain loci in $\qu$ containing $\Zmax$ and describe their cohomology. Our goal is to understand the explicit structure of highest-dimensional commutative subspaces of $\lendi (V)$ satisfying the $r$-spanning condition. We use the classical works of Schur and Jacobson \cite{J44} for the $r\geq \frac{d}{2}$ case, and prove some new results in the case $1< r < \frac{d+1}{2}$. We believe that they may have interesting applications for the study of varieties of commutative matrices.

\subsection{Maximal dimensional subspaces}

The matrices of the form $\langle pr_1(y) \rangle$ for $y \in \ZnrlwYmax$ are exactly those, which are maximal-dimensional among the $r$-spanning commutative spaces of matrices.

\begin{Lemma}
    Assume that $A \subset \lendi (V)$ is a maximal dimensional $r$-spanning, commutative subspace of matrices. Then it is equal to the subalgebra of $\lendi (V)$ it generates.
\end{Lemma}
\begin{proof}
    Assume that this were not the case, denote by $\tilde{A} \supsetneq A$ the subalgebra of $\lendi (V)$ generated by $A$. By assumption $\dime(\tilde{A}) > \dime(A)$. By distributivity and associativity of matrix multiplication, the elements of $\tilde{A}$ are pairwise commutative. Since $A \subset \tilde{A}$, the $r$-spanning condition is automatically satisfied.
\end{proof}
We remark on the following interpretation of the $1$-spanning condition.
\begin{Lemma}
    Let $A$ be a $1$-spanning, commutative subalgebra of $\lendi (V)$. Then $\dime(A)=d$.
\end{Lemma}
\begin{proof}
    Pick an element $v\in V$, such that $\langle v \rangle$ witnesses the $1$-spanning condition. There exist elements $a_1,\dots,a_{d} \in A$ with $a_1=Id$ such that $\{a_1v, \dots, a_{d}v\}$ form a basis of $V$, implying the lower bound. We claim that $a_1,\dots,a_{d}$ form a basis of $A$. Pick an element $a\in A$, there exist constants $\lambda_i$ for $i=1,\dots, d$ so that $(a-\sum\limits_{i=1}^d \lambda_ia_i)v=0$. Then for $j=1,\dots,d$
    \[ \left(a-\sum\limits_{i=1}^d \lambda_ia_i\right)\cdot (a_jv) = a_j\cdot \left(a-\sum\limits_{i=1}^d \lambda_ia_i\right)v=0. \]
    Thus $a = \sum\limits_{i=1}^d \lambda_ia_i$, proving the claim.
\end{proof}
Conversely, each $d$-dimensional $\kk$-algebra $A$ gives rise to a $1$-spanning space of commutative matrices in $End(V)$, whose elements are given by the action of $A$ on itself, via multiplication.\\

We introduce a special class of subspaces of $\lendi (V)$, which will be useful in describing the subspaces of interest of this section.

\begin{Definition} \label{Wkldef}
    Fix a basis of $V$. We denote by $W_{k,l} \subset \lendi (V)$ the subspace of all matrices whose only nonzero entries lie in the first $k$ rows and last $l$ columns.
\end{Definition}

We now recall a theorem of Schur and Jacobson on the maximal dimensional subspaces of $\lendi (V)$.

\begin{Theorem}[{\cite[Theorem 3]{J44}}]
\label{jac}
    Let $d>3$. Then if $A$ is a subalgebra of $\lendi (V)$ of maximal possible dimension, then $A$ is equivalent, milto conjugation, to $\langle Id \rangle \oplus A'$, where if $d=2k$, then $A' = W_{k,k}$, and if $d=2k+1$, then
    \[ A' = \begin{cases} W_{k+1,k} \\W_{k,k+1}
    \end{cases}. \]
\end{Theorem}

Since the subspaces $W_{d-k,k}$ are square-zero for all $1 \leq k < d$, the subalgebra generated by $W_{d-k,k}$ is equal exactly to $\langle Id \rangle + W_{d-k,k}$. Since the image of every matrix in $W_{d-k,k}$ lies in the span of the first $d-k$ basis vectors, it is clear that $W_{d-k,k}$ are $k$-spanning, and are not $l$-spanning for any $l < k$. This yields the following corollary.

\begin{Corollary}
    Let $d>3$ and $r\geq \frac{d}{2}$.
    \begin{enumerate}
        \item If $d=2k$, then a maximal dimensional commutative $r$-spanning subspace is, up to conjugation, equivalent to $\langle Id \rangle \oplus W_{k,k}$. In particular $l_{\max}(r) = \frac{d^2}{4}+1$.
        \item If $d=2k+1$, then a maximal dimensional commutative $r$-spanning subspace is, up to conjugation, equivalent to $\langle Id \rangle \oplus W_{k+1,k}$, or $\langle Id \rangle \oplus W_{k,k+1}$. In particular $l_{\max}(r) = \frac{d \cdot (d+1)}{4}+1$.
    \end{enumerate}
    Furthermore if $d=2k+1$ and $r=k$, then a maximal dimensional commutative $r$-spanning subspace is up to conjugation equivalent to $\langle Id \rangle \oplus W_{k+1,k}$.
\end{Corollary}

\subsection{\texorpdfstring{The value of $l_{\max}(r)$ for $1<r<\frac{d+1}{2}$}{The value of lmax}}

The corollary above does not cover the whole range of possible $r$. The next theorem allows us to understand the cases of $1<r<\frac{d+1}{2}$. In particular this implies that $d > 3$. To our knowledge the part $(b)$ of the theorem is new, while $(a)$ follows rather easily from already available results.

\begin{Theorem}
\label{sup}
    Let $1<r<\frac{d+1}{2}$ and let $A$ be an $r$-spanning commutative subalgebra of $\lendi (V)$. Then
    \begin{enumerate}
        \item[(a)] $\dime(A) \leq r \cdot (d-r) +1$.
        \item[(b)] If $\dime(A) = r \cdot (d-r) +1$, then $A$ is equivalent, up to conjugation, to $\langle Id \rangle + W_{d-r,r}$.
    \end{enumerate}
\end{Theorem}

In the proof we use a sequence of technical lemmas, which bound the length of an algebra $A$ as a module over itself. For any $\kk$-algebra $A$ and an $A$-module $M$ we have $\len(M) \leq \dime_{\kk}(M)$. However in our setting, where we assume that $\bar{\kk} = \kk$, equality holds.
\begin{Lemma}
    Let $A$ be a local, finite-dimensional $\kk$-algebra with maximal ideal $\mathfrak{m}$ and let $M$ be an $A$-module, finite dimensional as a $\kk$-vector space. Then $\len(M) = dim_{\kk}(M)$.
\end{Lemma}
\begin{proof}
    It is enough to prove that all simple $A$-modules are $1$-dimensional over $\kk$. Let $M \neq 0$ be such a module. Then by Nakayama's lemma $\mathfrak{m}M$ is a proper submodule of $M$, thus $\mathfrak{m}M=0$. Therefore $M$ is an $A/\mathfrak{m} \simeq \kk$-vector space, so
    \[ \len_A(M) = \len_{\kk}(M) = \dime_{\kk}(M). \qedhere \]
\end{proof}
\begin{Lemma} \label{nakajama}
    Let $A$ be a local ring with maximal ideal $\mathfrak{m}$ and let $M$ be an $A$-module. Then
    \[ \len(M) = \len(\mathfrak{m}M) + \mu(M), \]
    where $\mu(M)$ is the minimal number of generators of $M$ as an $A$-module.
\end{Lemma}
\begin{proof}
    By definition length is additive in short exact sequences. By the short exact sequence
    \begin{align*}
        0 \rightarrow \mathfrak{m}M \rightarrow M \rightarrow M/\mathfrak{m}M \rightarrow 0
    \end{align*}
    it is enough to prove that $\mu(M) = \len(M/\mathfrak{m}M)$. By Nakayama's lemma $\mu(M) = \mu(M/\mathfrak{m}M)$. But for a module over a local ring $A$ with the maximal ideal $\mathfrak{m}$ acting by zero (that it, an $A/\mathfrak{m}$-vector space) the length and the minimal number of generators agree with the dimension over the field $\kk = A/\mathfrak{m}$, so they are equal.
\end{proof}

\begin{Lemma}[{\cite[Lemma 2]{C93}}] \label{cow}
    Let $A$ be a commutative Artinian ring and let $M,N$ be $A$-modules, $N$ having finite length. Denote by $\mu(M)$ the minimal number of generators of $M$ as an $A$-module. Then,
    \begin{align*}
        \len(\Hom_A(M,N)) \leq \mu(M) \cdot \len(N).
    \end{align*}
\end{Lemma}
We include the proof from \cite{C93}, as analysing when equality holds is crucial for the proof of Theorem \ref{sup}.
\begin{proof} \label{prooflemmalength}
    Since the functor $\Hom_A(-,N)$ is left exact, a surjection $A^{\mu(M)} \twoheadrightarrow M$ yields an injection $\Hom_A(M,N) \hookrightarrow \Hom_A(A^{\mu(M)},N) \simeq A^{\oplus \mu(M)}$. The length is monotonic under taking submodules, which yields the statement.
\end{proof}

Observe that if we consider $V$ as a module over a commutative algebra $A \subset \lendi (V)$, then $\mu(V)$ is exactly the minimal number $r$ such that $A$ is $r$-spanning.

\begin{Lemma}\label{lemmapostac}
    Let $A \subset \lendi(V)$ be an indecomposable, commutative subalgebra. Then it is equivalent, up to a change of basis, to an algebra of the form $Id \oplus \mathfrak{m}$, where $\mathfrak{m}$ consists of strictly upper triangular matrices.
\end{Lemma}
\begin{proof}
    By Lie's theorem we can write $A$ in the strictly upper triangular form. It is enough to show that each element $B$ of $A$ is of the form $\lambda \cdot Id + B'$ for some $\lambda \in \kk$, $B'\in \lendi(V)$, where $B'$ is strictly upper triangular, or equivalently that all the eigenvalues of $B$ are the same.\\
    Assume that this is not the case. Consider eigenvalues $\lambda_1,\dots,\lambda_k$ of $B$ and denote by $K_i = \lker((B-\lambda_i)^d)$ for $i=1,\dots,k$. Then $V = K_1 \oplus \ldots \oplus K_k$. We show that the projections of $V$ to $K_i$ lie in $A$ for every $i=1,\dots,k$, contradicting indecomposability. Note that the subspaces $K_i$ are invariant under $A$, as
    \[ \underset{a\in A}{\forall}\ \underset{k\in K_i}{\forall} (B-\lambda_iId)^d \cdot (a\cdot k) = a\cdot (B-\lambda_iId)^d \cdot k = 0.\]
    Thus each element of $A$, in particular $B$, has a block structure with blocks corresponding to its restrictions to $K_i$ for $i=1,\dots,k$. Consider an element
    \[ B_{1,2} : =\left( \frac{-1}{(\lambda_2-\lambda_1)^d}((B-\lambda_1Id)^d - (\lambda_2-\lambda_1)^dId) \right)^d \in A. \]
     It has block form with blocks corresponding to its restrictions to $K_i$ for $i=1,\dots,k$, its restriction to $K_1$ is identity, and it acts on $K_2$ by $0$. Analogously we can obtain elements $B_{1,i}$ for $i=2,\dots,k$, which act by identity on $K_1$ and by $0$ on $K_i$. Then the product
     \[ B_1 := B_{1,2} \cdot \ldots \cdot B_{1,k} \]
     is exactly the projection to $K_1$. Analogously we can obtain projections to $K_2,\dots,K_k$, contradicting the irreducibility of $A$.
\end{proof}
The following lemma is important for understanding the case of $dim(A) = r \cdot (d-r) + 1$ in Theorem \ref{sup}.
\begin{Lemma} \label{bazza}
        Let $A = Id \oplus \mathfrak{m}$ be an indecomposable, $r$-spanning, commutative subalgebra of $\lendi(V)$, such that $\dim(\mathfrak{m} \cdot V) = d-r$. Then there exists a basis $v_1,\dots,v_d$ of $V$ such that
    \begin{enumerate}
        \item The elements $v_1,\dots,v_{d-r}$ form a basis of $\mathfrak{m}\cdot V$,
        \item The elements $v_{d-r+1},\dots,v_d$ generate $V$ as an $A$-module,
        \item All elements of $\mathfrak{m}$ written with respect to the basis $v_1,\dots,v_r$ are strictly upper triangular matrices.
    \end{enumerate}
    \end{Lemma}
    \begin{proof}
        By Lemma \ref{lemmapostac} for an indecomposable subalgebra $A$ of $\lendi (V)$ there exists some basis $w_1,\dots,w_d$ with respect to which $A$ is of the form $Id \oplus \mathfrak{m}$, where $\mathfrak{m}$ is strictly upper triangular. We can modify the basis of $V$ while preserving this form, by adding multiples of vectors with smaller indices to those with larger, so that $d-r$ of those vectors span $\mathfrak{m}\cdot V$. This can be seen by the following procedure. First pick an echelon form basis $u_1,\dots, u_{d-r}$ of $\mathfrak{m}\cdot V$ with respect to $w_d,\dots,w_1$, that is so that if
        \begin{align*}
            u_i = \sum\limits_{j=1}^d a_{i,j} w_j\ \text{ and }\ c_i = \max\{ j\in [d] \mid a_{i,j} \neq 0\}
        \end{align*}
        then the sequence $c_i$ is strictly decreasing. Then add to a vector $w_{c_1}$ multiples of $w_j$ for $j<c_1$ to change $w_{c_1}$ into a nonzero multiple of $u_1$. Do the same for $w_{c_2},\dots,w_{c_{d-r}}$, until getting a basis $w_1',\dots,w_d'$, with respect to which $\mathfrak{m}$ consists of strictly upper triangular matrices and $w_{c_i}' = u_{i}$ for $i=1,\dots,d-r$. Then move the vectors $w_{c_{d-r}}',\dots,w_{c_{1}}'$ to the first $d-r$ positions, whilst keeping the others in the same order. Since $\mathfrak{m} \cdot V$ is an invariant subspace of $V$, we can consider $\mathfrak{m}$ restricted to it. Then, by Lie's theorem, there exists a choice of basis $u_1',\dots,u_{d-r}'$ of $\mathfrak{m} \cdot V$ with respect to which all elements of the Lie algebra $\mathfrak{m}$ are strictly upper triangular. Then letting $v_i := u_i'$ for $i=1,\dots,d-r$ and setting the rest of $v_i$ to be $\{w_j\}_{j\in [d]\backslash \{c_1,\dots,c_{d-r}\}}$, with the preserved order of indices, we get a desired basis.
    \end{proof}

Lastly we prove a technical result, which allows us to restrict our attention to the local case.
\begin{Lemma}\label{sumaprostaalgebr}
    Let $A = A_1 \oplus A_2$ be a subalgebra of $\lendi(V)$. For $i=1,2$ let $r_i$ be the smallest number such that $A_i$ is $r_i$-spanning for $V_i=\imag(A_i)$. Then for the smallest possible value of such $r$, such that $A$ is $r$-spanning for $V=\imag(A)$ is equal to $r := \max\{r_1,r_2\}$.
\end{Lemma}
\begin{proof}
    Clearly $r \geq r_1,r_2$, as for $v_1,\dots,v_r \in V$ such that $A$ is $r$-spanning, we have that 
    \[ \imag(A_1) = A_1 \cdot \langle v_1 , \dots, v_r \rangle = A_1 \cdot \langle (1,0) \cdot v_1 , \dots, (1,0) \cdot v_r \rangle . \]
    Furthermore, if for $i=1,2$ \[ v_1^i,\dots,v_{r_1}^i \]
    witness the $r_i$-spanning condition for $A_i$, then
    \begin{gather*}
        A \cdot \langle (v_1^1,v_1^2) , \dots, (v_{\max\{r_1,r_2\}}^1, v_{\max\{r_1,r_2\}}^2) \rangle  = A_1 \cdot \langle v_1^1 , \dots, v_{r_1}^1 \rangle \oplus A_2 \cdot \langle v_1^2 , \dots, v_{r_1}^2 \rangle = \imag(A_1) \oplus \imag(A_2) = V
    \end{gather*}
    where for $i=1,2$, $j > r_i$ we set $v_j^i = 0$. Thus $r \leq \max\{r_1,r_2\}$.
\end{proof}

\begin{proof}[Proof of Theorem \ref{sup}]
    Assume first that $A$ is an indecomposable subalgebra of $\lendi(V)$, by Lemma \ref{lemmapostac} it has to be equivalent, up to conjugation, to a subalgebra of the form $\langle Id \rangle + \mathfrak{m}$ with $\mathfrak{m}$ a subspace consisting of strictly upper triangular matrices. Since $A$ is finite dimensional, it is Artinian, note that as $\mathfrak{m}$ is nilpotent, it is the unique maximal ideal of $A$. Apply Lemma \ref{cow} to $M=V$ and $N = \mathfrak{m}V$ treated as an $A$-module. Then, $\mathfrak{m}$ injects into $\Hom_A(V,\mathfrak{m}V)\subset \lendi (V)$
    \begin{align*}
        \len(\mathfrak{m}) \leq \len(\Hom_A(V,\mathfrak{m}V)) \leq \mu(V) \cdot \len(\mathfrak{m}\cdot V).
    \end{align*}
    By the remark above $\mu(V)\leq r$. Since $A$ is a local algebra with a maximal ideal $\mathfrak{m}$, by Lemma \ref{nakajama} we have $\len(\mathfrak{m}V) = \len(V)-\mu(V)$. Thus for $r<\frac{d+1}{2}$
    \begin{align*}
        \mu(M) \cdot \len(\mathfrak{m}V) = \mu(V) \cdot (\len(V)-\mu(V)) = \mu(V) \cdot (d-\mu(V)) \leq r \cdot (d-r).
    \end{align*}
    since $\mu(V) \leq $ Thus $\dime(A) = \dime(\mathfrak{m}) +1 \leq r \cdot (d-r) +1$.\\
    In the general case, since $A$ is finite dimensional over $\kk$, it can be written as a direct sum $A_1 \oplus \ldots \oplus A_k$ of indecomposable, local algebras. Denote $r_i = \mu(A_i)$ and $d_i = \dime(\imag(A_i))$. Then by Lemma \ref{sumaprostaalgebr}, $r = \underset{i \in [k]}{\max}\ r_i$ and $d = \sum\limits_{i=1}^{k}d_i$. Applying the above consideration to each of the components of $A$ yields:
    \begin{align} \label{wymAwynik}
        \dime(A) \leq \sum\limits_{i=1}^{k} (r_i \cdot (d_i-r_i) +1).
    \end{align}
    Take $i,j \in [k]$ such that $r_i \geq r_j$. Then
    \begin{gather} \label{pomocniczanier}
        \big(r_i\cdot (d_i-r_i)+1\big)+\big(r_j\cdot (d_j-r_j) +1\big) 
        \leq (r_i \cdot (d_i+d_j-r_i-r_j)) + r_i \cdot r_j + 1 
        = \\ =(r_i \cdot (d_i+d_j - \max\{r_i,r_j\}) + 1
    \end{gather}
    and equality can occur only if $r_i \cdot r_j = 1$, thus $r_i=r_j=1$. Note that $r_i=1$ can hold only if $\dim(A_i) = d_i$. Here we need not assume that $A_i$ is indecomposable.
    Applying inductively the inequality above to Inequality \ref{wymAwynik}, we see that
    \begin{align} \label{rozkladalnanier}
        \dime(A) \leq ( \underset{i \in [k]}{\max}\ r_i) \cdot \left(\sum\limits_{i=1}^{k} d_i - \underset{i \in [k]}{\max}\ r_i \right) +1 = r\cdot (d-r) +1
    \end{align}
    proving $(a)$.\\
    Now consider the case where $\dime(A) = r\cdot (d-r) +1$. Then all the inequalities above have to be equalities. Therefore we can assume $\mu(M)=r$. Equality in Inequality \ref{rozkladalnanier} holds only if the two sides of Inequality \ref{pomocniczanier} are always equal. But this is only possible if for each $i \in [k]$, $r_i=1$, thus $\dim(A) \leq d < r \cdot (d-r) +1$, contradiction. Thus we can assume that $A$ is indecomposable.\\
    Pick a basis $v_1,\dots,v_d$ of $V$, such that $A$ satisfies the conditions of Lemma \ref{bazza}. We can then fix a surjection $A^{r} \rightarrow V$ that maps the $i$-th basis element of the free module $A^r$ to $v_{d-r+i}$ for $i=1,\dots,r$. In order for the injection
    \[\Hom_A(V,\mathfrak{m}V) \hookrightarrow \Hom_A(A^{\mu(M)},\mathfrak{m}V)\]
    described as in Lemma \ref{cow} to be an isomorphism, specifying any $r$ elements of $\mathfrak{m}V$ yields a unique homomorphism of $A$-modules from $V$ to $\mathfrak{m}V$, such that the fixed $r$ generators of $V$ map to the specified $r$ elements of $\mathfrak{m}V$. Furthermore since the injection $\mathfrak{m} \hookrightarrow \Hom_A(V,\mathfrak{m}V)$ must be an isomorphism, the lifting is given by an element of $\mathfrak{m} \subset \lendi (V)$. Consider a subspace $W \subset V$ generated by $v_{d-r+1},\dots,v_{d}$. Define the canonical basis
    \[ e_{i,j} \text{ of } \Hom_{\kk}(W,\mathfrak{m}V) \text{ for }i=1,\dots,d-r,\ j=d-r+1,\dots,d \]
    by $e_{i,j}v_k = \delta_{j,k}v_i$, where $\delta_{j,k}$ is the Kronecker's delta. Their extensions to $A$-module homomorphisms define a basis of $Hom_A(A^{r},\mathfrak{m}V)$. Thus we can choose a basis $\tilde{e}_{i,j}$ of $\mathfrak{m}$ lifting $e_{i,j}$ under the aforementioned composition. We show that the space $\mathfrak{m}$ is square-zero by verifying it on the basis $\tilde{e}_{i,j}$ of $\mathfrak{m}$ and the basis $v_1,\dots,v_{d-r}$ of $\mathfrak{m}\cdot V$. Assume, for the sake of contradiction, that there exist $k,l$ such that
    \[ \tilde{e}_{k,l}v_j \neq 0 \ \text{ for some } j\leq d-r.\]
    Pick an $s\geq d-r$ such that $l\neq s$, this is possible since $r>1$. By the commutativity condition of endomorphisms in $A$, we have
   \begin{align*}
       \tilde{e}_{k,l} \cdot \tilde{e}_{j,s} - \tilde{e}_{j,s}\cdot \tilde{e}_{k,l} = 0.
   \end{align*}
   Let us evaluate the operator above on $v_s$. We have $\tilde{e}_{j,s}v_s = v_j$ and $\tilde{e}_{k,l}v_s =0$, since $l\neq s$. Thus
   \begin{align*}
       (\tilde{e}_{k,l} \cdot \tilde{e}_{j,s} - \tilde{e}_{j,s}\cdot \tilde{e}_{k,l})(v_s) = \tilde{e}_{k,l} \cdot  v_j \neq 0
   \end{align*}
   contradicting our assumption. Thus indeed $\mathfrak{m}$ is square-zero, and the endomorphisms $\mathfrak{m} \subset End(V)$ written in the basis $v_1,\dots,v_d$ lie in the space of matrices $W_{d-r,r}$, which was to be proven in part $(b)$ of Theorem \ref{sup}.
\end{proof}

\subsection{\texorpdfstring{The $r$-spanning subspaces for $d\leq 2$}{The spanning subspaces for}}

Combined, the theorems above give a description of maximal dimensional $r$-spanning subspaces for all $r>1$ and $d>3$. We now describe the classification for $d \leq 2$. We will not use the case $d=3$ later, thus we omit it for the sake of conciseness. In the case $r=1$ the scheme $\qu$ is exactly $\Hilb_d(\mathbb{A}^n)$. Since for any algebra $A$, $\lendi_A(A) \simeq A$, all the algebras have a $d$-dimensional space of endomorphisms and the classification of $1$-spanning commutative subspaces reduces to the classification of finite algebras of length $d$. For the classification in the cases $d\leq 6$ see \cite{B08}, it is well known that there are infinitely many isomorphism classes over $\mathbb{C}$ in the case $d\geq 7$ \cite{S56}. Nevetherless, the work of \cite{HJ25} describes, using the techniques of $\mathbb{A}^1$-homotopy theory, the cohomology of the colimit $\Hilb_d(\mathbb{A}^{\infty}) := \underset{n\rightarrow \infty}{\colim}\Hilb_d(\mathbb{A}^n)$, and it is equal to the cohomology of the Grassmanian $\Grass(d-1,\infty)$ of $(d-1)$-dimensional subspaces in $k^{\oplus \infty}$.\\

We start our description with the trivial case $d=1$. Then $\lendi (V)$ is given by multiplication by scalars and $\dime(\lendi (V))=1$. Thus there is a unique $1$-dimensional ($r$-spanning) commutative subspace.\\

Consider now the case $d=2$. Assume that $\kk$ is an algebraically closed field. We claim the following holds:

\begin{Theorem} \label{wymdwamacierze}
    Let $V$ be of dimension $d=2$ and $A$ a commutative, $r$-spanning subspace of $\lendi (V)$ for $r=1$ or $r=2$. Then the maximal possible dimension of $A$ is $2$ and if $\dime(A)=2$, then $A$ is equivalent, up to a change of basis, to one of the following spaces of matrices:
    \begin{enumerate}
        \item Diagonal $2\times 2$ matrices,
        \item $\langle Id\rangle \oplus W_{1,1}$, that is, upper triangular $2\times 2$ matrices with the same element on the diagonal.
    \end{enumerate}
\end{Theorem}
\begin{proof}
    If $A$ contains any matrix which is not a multiple of identity, then it is necessarily $1$-spanning. Thus it is a space of endomorphism of a $\kk$-algebra, which over an algebraically closed field $\kk$ correspond to one of the spaces above.
\end{proof}

\section{Computation of the cohomology of the maximal dimensional loci} \label{sec3coh}

In this section we introduce loci of $\qu$ which contain the loci $\Zmax$ and calculate their cohomology. First consider the subspaces $W_{d-k,k}$ with only nonzero entries in the first $d-k$ rows and last $k$ columns we defined in Definition \ref{Wkldef}. Any subspace $A$ matrices satisfying $A^2=0$ for which
\[ \dime\left(\bigcap_{a\in A}\lker(a)\right) \geq d-k \text{ and } \dime\left(\sum_{a\in A}\imag(a)\right) \leq d-k \]
is, up to conjugation, equivalent to a subspace of $W_{d-k,k}$, that is by setting the first $d-k$ basis vectors to be arbitrary ones such that their span contains $\sum_{a\in A}\imag(a)$. Conversely, for any $k,l$ the space $W_{d-k,k}$ satisfies
\[ \dime\left(\bigcap_{a\in W_{k,l}}\lker(a)\right) \geq d-k \text{ and } \dime\left(\sum_{a\in W_{k,l}}\imag(a)\right) \leq d-k. \]

\begin{Definition}
    We define a subset $\tilde{S}_k$ of the traceless locus $\Yhat$ as the set of points $y$ with $pr_1(y) = (X_1,\dots,X_n)$ for which the following conditions are satisfied:
    \begin{align*}
        \dime\left(\bigcap\limits_{i=1}^n\lker(X_i)\right) \geq d-k,\ \ \dime\left(\sum\limits_{i=1}^n \imag(X_i)\right) \leq d-k, \ \ X_i \cdot X_j = 0 \text{ for }i,j=1,\dots,n .
    \end{align*}
    Since this is an invariant closed subset, its image $S_k := \pi(\tilde{S}_{k})$ is a closed subset of $\Qnrd \subset \qu$. 
\end{Definition}

We now define locally closed subsets $R^{r,n}$ of $\Qnrd$ such that the product $\mathbb{A}^n \times R^{r,n}$ contains $\Zmax$. Their geometry is the main object of study of this section. We compute their cohomology, giving us a better intuition for what could be the cohomology of the Quot scheme of an infinite affine space. The cohomology of $\mathbb{A}^n \times X$ is isomorphic to the cohomology of $X$ for any algebraic variety $X$. Thus to compute the cohomology of the loci $\mathbb{A}^n \times R^{r,n}$ in $\qu$, whose complement has codimension diverging to infinity with $n$, we can equivalently consider the loci $R^{r,n}$ in $\Qnrd$.

\begin{Definition}
    \begin{enumerate}
        \item Let $1<r<\frac{d+1}{2}$. Define $R^{r,n}$ to be $S_r$.
        \item Let $d=2k$ and $r\geq k$. Define $R^{r,n}$ to be the subset of $S_k$ defined by the condition
        \[ \dime\left(\sum\limits_{i=1}^n \imag(X_i)\right) = k. \]
        \item Let $d=2k+1$ and $r\geq k+1$. Define $R_1^{r,n}$ to be the subset of $S_{k}$ defined by the condition
        \[ \dime\left(\sum\limits_{i=1}^n \imag(X_i)\right) = k+1 \]
        and $R_2^{r,n}$ to be the subset of $S_{k+1}$ defined by the condition
        \[ \dime\left(\sum\limits_{i=1}^n \imag(X_i)\right) = k. \]
        Let $R^{r,n} := R_1^{r,n} \cup R_2^{r,n}$.
    \end{enumerate}
\end{Definition}

Note that we impose additional conditions in the cases $r \geq \frac{d}{2}$, which amount to replacing a lower bound for the sum of images of $X_i$ with an equality. Because of this the sets are not closed as in the $1<r<\frac{d+1}{2}$ case, but merely locally closed. The reason  for this condition is that for $r \geq \frac{d}{2}$ the sets defined without this condition (that is the closure of $R^{r,n}$) would be quite badly behaved and singular. Assume for simplicity that $d=2k$. It can be checked that the the singularity type is the same at every locally closed component $\dime(\sum\limits_{i=1}^n \imag(X_i)) = l$ for $0\leq l <k$ of the closure $\overline{R^{r,n}}$ and corresponds to a linear section of some determinental variety. However in the $1<r<\frac{d+1}{2}$ the condition on $\sum\limits_{i=1}^n \imag(X_i)$ being of maximal possible dimension is automatically satisfied, as $\dime(\imag(S_0^{\oplus r})) \leq r$ and $\imag(S^{\oplus r}_{\geq 1}) \subset \sum\limits_{i=1}^n \imag(X_i)$.\\

We now determine the cohomology of the $R^{r,n}$. We begin with the simplest case $1<r<\frac{d+1}{2}$.

\begin{Theorem}
    Let $1<r<\frac{d+1}{2}$. Then $R^{r,n} \simeq \Grass(n\cdot r,d-r)$.
\end{Theorem}
\begin{proof}
    For a point $S^{\oplus r} \twoheadrightarrow M$ of $R^{r,n}$ consider the vector space map given by $(S^{\oplus r})_1 \longrightarrow \imag((S^{\oplus r})_1)$ induced by the quotient. As stated above, since $\dime(\imag((S^{\oplus r})_0)) \leq r$ we automatically have $\dime(\imag(S^{\oplus r}_{\geq 1})) = d-r$. Since by the square-zero assumption $\imag((S^{\oplus r})_2) = 0$, this implies that the construction gives a point of the quotient Grassmanian $\Grass(n\cdot r,d-r)$. We claim it is an isomorphism of schemes. Since $R^{r,n}$ is a closed subscheme of a reduced scheme it is reduced, so it is enough to verify the claim on closed points. There it follows since $\imag((S^{\oplus r})_{\geq 2}) = 0$ and the quotient is injective on $S_0^{\oplus r}$, so the assignment uniquely determines the kernel. 
\end{proof}
\begin{Corollary}
    Let $1<r<\frac{d+1}{2}$. Then $R^{r,\infty} \simeq \Grass(d-r,\infty)$.
\end{Corollary}
\begin{proof}
    Under the embedding $\iota_n$ a quotient of $(\kk[x_1,\dots,x_n]^{\oplus r})_1$ is mapped to a quotient of $(\kk[x_1,\dots,x_{n+1}]^{\oplus r})_1$ which is the same under the restriction to $(\kk[x_1,\dots,x_n]^{\oplus r})_1$ and zero on $X_{n+1}(\kk[x_1,\dots,x_n]^{\oplus r})_0$. This agrees exactly with the composition of the sequence of embeddings $r$ embeddings in the direct limit defining $\Grass(d-r,\infty)$, so the colimits agree.
\end{proof}

As we see, the loci $R^{r,n}_d$ are particularly well-behaved in the cases $1< r < \frac{d+1}{2}$.
Let us now move our attention to the case $r \geq \frac{d}{2}$.

\begin{Theorem} \label{proofkohRrn}
    Let $d=2k$ and $r\geq k$. Then $H^*(R^{r,n}) \simeq H^*(\Grass(r,k)) \otimes H^*(\Grass(n \cdot k,k))$.
\end{Theorem}
\begin{proof}
    First we show that we can restrict to the locus $R^{r,n}_{\circ}$, defined by the subspace of surjections satisfying the condition
    \[ \dime\left( \imag\big((S^{\oplus r})_0\big) \right) = k . \]
    It is a closed subset of $R^{r,n}$, as on $R^{r,n}$ the condition above is equivalent to $\dime\left( \imag\big((S^{\oplus r})_0\big) \right) \leq k$. There is a morphism $\beta: R^{r,n} \twoheadrightarrow R^{r,n}_{\circ}$, defined on $\kk$-points by sending a surjection
    \[ \phi: S^{\oplus r} \twoheadrightarrow M \ \text{ to } \ \tilde{\phi}: S^{\oplus r} \twoheadrightarrow \tilde{M}. \]
    Here, as a vector space,
    \[ \tilde{M} \simeq  M/\imag(S^{\oplus r}_{\geq 1}) \oplus \imag(S^{\oplus r}_{\geq 1}). \]
    $S_1$ acts on the second summand by $0$, while the image of $M/\imag(S^{\oplus r}_{\geq 1})$ under the action of $S_1$ is contained in the second summand $\imag(S^{\oplus r}_{\geq 1})$ and is induced by the action of $S$ on $M$,
    \[ S_1 \times M \rightarrow \imag(S^{\oplus r}_{\geq 1}) \]
    using the fact that $S_2$ acts trivially on $M$. The morphism to the linear subspace $M/\imag(S^{\oplus r}_{\geq 1})$ is the quotient of $\phi$ restricted to $(S^{\oplus r})_0$, and we extend it to an $S$-module homomorphism via the action of $S$ on $\tilde{M}$ described above.\\
    
    Consider coverings of $R^{r,n}$ and $R^{r,n}_{\circ}$ by open subsets of $\phi: S^{\oplus r} \twoheadrightarrow M$ satisfying
    \[ \imag\left(\langle e_{s_1},\dots,e_{s_k},X_{i_{1}}e_{s_{j_1}}, \dots, X_{i_{k}}e_{s_{j_k}}\rangle\right) = M \]
    for fixed subsets
    \[ \{s_1,\dots,s_k\} \subset [r],\ \{(i_1,j_1),\dots,(i_k,j_k)\} \subset [n]\times [k]. \]
    Choose an element
    \[ [\phi: S^{\oplus r} \twoheadrightarrow M] \in R^{r,n}, \]
    lying in the open subset above.
    For $i=1,\dots,n$, $j=1,\dots,k$ such that $(i,j)\notin \{(i_1,j_1),\dots,(i_k,j_k)\}$, write the image of $X_ie_{s_j}$ under $\phi$ as a linear combination of elements
    \[ \{X_{i_{1}}e_{s_{j_1}}, \dots, X_{i_{k}}e_{s_{j_k}}\}. \]
    For $i\in [r]$ such that $i\notin \{ s_1,\dots,s_k \}$ write the image of $e_i$ under $\phi$ as a linear combination of elements
    \[e_{s_1},\dots,e_{s_k},\ X_{i_{1}}e_{s_{j_1}}, \dots, X_{i_{k}}e_{s_{j_k}}. \]
    Conversely, any choice of coefficients in the linear combinations above uniquely defines an element of $R^{r,n}$. Since a $\kk$-point of the open subset of $R^{r,n}$ is uniquely specified by assigning values in $\kk$ to a fixed number of free variables, this open subset is isomorphic to an affine space of dimension equal to the number of those variables. We specify the expressions of $n\cdot k -k$ elements of the form $X_ie_j$ in the $k$-element basis $X_{i_1}e_{s_{j_1}}, \dots, X_{i_k}e_{s_k}$. Additionally, we specify the coefficients of the $k$ basis elements $e_{s_1},\dots,e_{s_k}$ and the $k$ basis elements $X_{i_1}e_{s_{j_1}}, \dots, X_{i_k}e_{s_k}$ in the expressions for the images of the $r-k$ elements $e_i$, for $i \notin \{s_1,\dots,s_k\}$. Thus the open set we are considering is isomorphic to the affine space
    \[ \aff^{(nk - k)\cdot k} \times \aff^{(r-k)\cdot k} \times \aff^{(r-k)\cdot k} .\]
    An element $\phi$ lies in $R^{r,n}_{\circ}$ if and only if the coefficients of $\{X_{i_{1}}e_{s_{j_1}}, \dots, X_{i_{k}}e_{s_{j_k}}\}$ in the expression in basis of the images of $e_i$ for $i\notin \{ s_1,\dots,s_k \}$ are zero. In coordinates $\beta$ is given by a linear projection to the subspace $\aff^{n\cdot k - k} \times \aff^{(r-k)\cdot k} \times \{0\}$. Thus $\beta : R^{r,n} \rightarrow R^{r,n}_{\circ}$ is a fiber bundle with an $(r-k)\cdot k$-dimensional affine space as the fiber. In particular
    \[ H^*(R^{r,n}) \simeq H^*(R^{r,n}_{\circ}). \]
    The advantage of considering $R^{r,n}_{\circ}$ is that it is a projective variety, allowing us to employ the classical version of the Białynicki-Birula decomposition \cite[Theorem 4.2]{BCM02}.\\


    Consider the morphism sending a point $\phi: S^{\oplus r} \twoheadrightarrow M$ of $R^{r,n}_{\circ}$ to the kernel of its restriction to $(S^{\oplus r})_0$, that is $ker\big((S^{\oplus r})_0 \rightarrow M\big)$. Since
    \[ \dime\big(\imag(S^{\oplus r}_{\geq 1})\big) = \dime\left(\sum\limits_{i=1}^n \imag(X_i)\right) = k, \]
    we have $\dime( M/\imag(S^{\oplus r}_{\geq 1})) = d-k = k$. Thus the construction above gives us a morphism $\alpha$ from $R^{r,n}_{\circ}$ to the quotient Grassmanian $\Grass(r,k)$.\\
    
    We would like to understand what is the fiber of this morphism over a point $\overline{\phi} \in \Grass(r,k)$. Let $K_{\phi} := \lker(\overline{\phi})$. Consider a $k$-dimensional subspace $U \subset (S^{\oplus r})_0$ not intersecting $K_{\phi}$. Since $(S^{\oplus r})_{\geq 2} \subset \lker(\phi)$, the kernel of $\phi$ is uniquely determined by the kernel of $S^{\oplus r}_{\leq 1} \rightarrow M$. Since we are considering points in the fiber over a fixed $\overline{\phi}$, it is enough to determine the kernel of $\phi$ restricted to $S^{\oplus r}_1$. By the assumption
    \[ \dime(\sum\limits_{i=1}^n \imag(X_i)) = k \text{ and } \dime \left( \sum\limits_{i=1}^n \imag(X_i) \right) = k \]
    determining such a kernel is equivalent to determining a quotient given by restricting the morphism $\phi: (S^{\oplus r})_1 \twoheadrightarrow \imag(S_1^{\oplus r})$ to the subspace $S_1 \cdot U$. This corresponds to a point in the quotient Grassmanian $\Grass(n \cdot k, k).$\\
    
    We see that $\alpha^{-1}(\overline{\phi}) \simeq \Grass(n\cdot k, k)$ for any $\overline{\phi} \in \Grass(r,k)$. Thus $R^{r,n}_{\circ}$ is smooth, as it is a fiber bundle with smooth fiber $\Grass(n \cdot k, k)$ and smooth base $\Grass(r,k)$.\\

    Consider the action of $\Gm$ on $R^{r,n}_{\circ}$ given by the coordinate action with weights
    \[ \lambda_1, \ldots , \lambda_r \text{ on } (S^{\oplus r})_0 \text{ and } \gamma_1,\dots,\gamma_n \text{ on } \aff^n, \]
    satisfying
    \[ 0<\lambda_1< \ldots < \lambda_r \text{ and } \gamma_1<\dots<\gamma_n . \]
    Assume furthermore that
    \[ \gamma_{i}-\gamma_{i-1}> \lambda_r  \ \text{ for } i=2,\dots,n . \]
    The fixed points of this action are (non-uniquely) indexed by a tuple
    \[  (s_1,\dots,s_k)\ \]
    and a subset 
    \[ \{ (i_1,j_1), \dots, (i_k,j_k) \} \subset [n]\times [k]. \]
    The choice of such a tuple and subset corresponds to a surjection $\phi: S^{\oplus r} \twoheadrightarrow M$, where the images of
    \[ \{ e_{s_1},\dots,e_{s_k}, X_{i_1}e_{s_{j_1}},\dots,X_{i_k}e_{s_{j_k}} \} \]
    span $M$, and all the other elements of the form $e_j$, $X_ie_j$ map to zero under $\phi$. This gives a one-to-one correspondence under the additional assumption that $s_1<\ldots<s_k$, but we omit this condition in order to simplify the notation in the proof, always choose a suitable representative.\\

    We can consider a similar $\Gm$-action on the product of Grassmanians
    \[ \Grass(r,k) \times \Grass(n \cdot k,k). \]
    We view the Grassmanianns as parametrizing the quotients of an $r$-dimensional vector space with basis $e_1,\dots,e_r$, and an $n\cdot k$-dimensional vector space with basis $e_{i,j}$, where $i=1,\dots,n$, $j=1,\dots,k$, respectively. The $\Gm$-action on the product is defined by assigning weights $\lambda_1, \dots, \lambda_r$ to the basis elements $e_1,\dots,e_r$, and weights $\gamma_i + \lambda_j$ on the basis elements $e_{i,j}$ of the $n\cdot k$-dimensional ambient space of $\Grass(n\cdot k,k)$. Note that the action of $\Gm$ on $\Grass(n\cdot k,k)$ factor doesn't depend on the weights $\lambda_i$ for $i>k$. The fixed points of this action can be indexed by the same data (consisting of a tuple and a subset) as the fixed points of the  $\Gm$-action on $R^{r,n}_{\circ}$. They correspond to a pair of coordinate surjections, where the subspaces
    \[ \langle e_{s_1},\dots,e_{s_k} \rangle \text{ in } \Grass(r,k) \text{ and } \langle e_{i_1,j_1},\dots,e_{i_k,j_k} \rangle \text{ in } \Grass(n\cdot k,k) \]
    are mapped isomorphically in the respective quotients, while all the other coordinate vectors are sent to zero.\\

    Since $R^{r,n}_{\circ}$ is smooth and projective, we can apply to it the classical Białynicki-Birula decomposition. We claim that the numbers of affine cells of a fixed dimension in the decompositions of $R^{r,n}_{\circ}$ and $\Grass(r,k) \times \Grass(n \cdot k,k)$ under the respective $\Gm$-actions coincide. It suffices to show, by Lemma \ref{styczne}, that at each fixed point, the dimension of the subbundle of the tangent bundle on which $\Gm$ acts with a positive character is the same for $R^{r,n}_{\circ}$ as it is for $\Grass(r,k) \times \Grass(n \cdot k,k)$. To that end, we first find a set of invariant curves on $R^{r,n}_{\circ}$ and on $\Grass(r,k) \times \Grass(n \cdot k,k)$ such that, at each fixed point, their tangent vectors form a basis of the tangent space in the respective varieties. We pair the curves above, tangent at corresponding fixed points of $R^{r,n}_{\circ}$ and $\Grass(r,k) \times \Grass(n \cdot k,k)$, in such a way that $\Gm$ acts with a positive character on a curve in $R^{r,n}_{\circ}$ if and only if it does so on the corresponding curve in $\Grass(r,k) \times \Grass(n \cdot k,k)$. This proves our claim, as the span of the tangent vectors to invariant curves with positive $\Gm$-character coincides with the subbundle of the tangent bundle on which $\Gm$ acts with positive characters.\\
        
    We consider two types of $\Gm$-invariant curves connecting fixed points on $R^{r,n}_{\circ}$ and $\Grass(r,k)\times \Grass(n\cdot k,k)$.
    \begin{enumerate}
        \item[(a)] A unique invariant degree one curve connecting
    \[ (s_1,\dots,s_k), \{(i_1,j_1),\dots,(i_k,j_k)\}\ \text{ and }\ (s_1,\dots,s_k'),\{(i_1,j_1),\dots,(i_k,j_k)\}.\]
     On $\Grass(r,k) \times \Grass(n\cdot k,k)$ it is constant on $\Grass(n\cdot k,k)$. The torus $\Gm$ acts on the curve with the character $\lambda_{s_k} - \lambda_{s_k'}$ on both varieties.
        \item[(b)] Consider the unique invariant degree one curve connecting the fixed points
    \[ (s_1,\dots,s_k),  \{(i_1,j_1),\dots,(i_k,j_k)\} \ \text{ with } \ (s_1,\dots,s_k),\{ (i_1,j_1),\dots,(i_k',j_k') \} , \]
    where we picked a representative of the fixed point such that $s_1<\ldots<s_k$.
    On $R^{r,n}_{\circ}$ the torus acts with the character $\lambda_{s_{i_k}}-\lambda_{s_{i_k}'}+\gamma_{j_k}-\gamma_{j_k'}$. On $\Grass(r,k) \times \Grass(n\cdot k,k)$ it is constant on $\Grass(r,k)$ and $\Gm$ acts with the character $\lambda_{i_k}-\lambda_{i_k'}+\gamma_{j_k}-\gamma_{j_k'}$.
    \end{enumerate}

    We show that the pairs of curves on $R^{r,n}_{\circ}$ and $\Grass(r,k) \times \Grass(n\cdot k,k)$ carry the character of the same sign. For curves of type $(a)$ this is clear, as the characters are the same. For curves of type $(b)$, since $|\gamma_i-\gamma_j|>\lambda_r$ for any $i\neq j$, they have the same sign if $j_k \neq j_k'$. Thus we assume that $j_k=j_k'$. Then the $\Gm$-characters are $\lambda_{s_{i_k}}-\lambda_{s_{i_k'}}$ on $R^{r,n}_{\circ}$ and $\lambda_{i_k}-\lambda_{i'_k}$ on $\Grass(r,k) \times \Grass(n\cdot k,k)$. As we chose a representative with $s_1<\dots<s_k$, one character is positive if and only if the other is.\\

    Thus the affine cell decompositions of $R^{r,n}_{\circ}$ and $\Grass(r,k) \times \Grass(n\cdot k,k)$ are the same, so there is a non-canonical isomorphism of cohomology groups    
    \begin{align*}
        H^*(R^{r,n}_{\circ}) \simeq H^*(\Grass(r,k) \times \Grass(n\cdot k,k)) \simeq H^*(\Grass(r,k)) \otimes H^*(\Grass(n\cdot k,k)),
    \end{align*}
    which was to be proven, since $H^*(R^{r,n}) \simeq H^*(R^{r,n}_{\circ})$.
\end{proof}
\begin{Corollary}
    Let $d=2k$ and $r\geq k$. The cohomology of the colimit $R^{r,\infty} = \underset{n\rightarrow \infty}{\colim}\ R^{r,n}$ is isomorphic to $H^*(\Grass(r,k)) \otimes H^*(\Grass(k,\infty))$.
\end{Corollary}
\begin{proof}
    It is clear that under each of the embeddings $\iota_n$ our construction is identity on $\Grass(r,k)$ and on the fibers of $\alpha$ maps the Grassmanian $\Grass(n \cdot k,k)$ into $\Grass((n+1) \cdot k,k)$ by extending the quotient to be zero on the last $k$ coordinate vectors (corresponding to the restriction $\phi_{\mid X_nU}$). This implies that under pullback the embeddings of $R^{r,n}$ into $R^{r,n+1}$ induce the same maps as in the direct limit of topological spaces defining $\Grass(\infty,k) \times \Grass(r,k) = \underset{n\rightarrow \infty}{\colim} \Grass(n\cdot k,k) \times \Grass(r,k)$. Since cohomology commutes with filtered colimits, we have
    \[ H^*(R^{r,n}) \simeq H^*(\Grass(r,k)) \otimes H^*(\Grass(k,\infty)) . \qedhere \]
\end{proof}

\begin{Theorem}
    Let $d=2k+1$ and $r\geq k+1$. Then
    \begin{align*}
        H^*(R^{r,n}) \simeq H^*(\Grass(r,k)) \otimes H^*(\Grass(n \cdot k,k+1)) \oplus H^*(\Grass(r,k+1)) \otimes H^*(\Grass(n \cdot (k+1),k)).
    \end{align*}
\end{Theorem}
\begin{proof}
    First we note that
    \begin{align*}
        H^*(R^{r,n}_1) \simeq H^*(\Grass(r,k)) \otimes H^*(\Grass(n \cdot k,k+1))
    \end{align*}
    and
    \begin{align*}
        H^*(R^{r,n}_2) \simeq H^*(\Grass(r,k+1)) \otimes H^*(\Grass(n \cdot (k+1),k)).
    \end{align*}
    Since the proof is completely analogous to the proof of Theorem \ref{proofkohRrn}, we omit it. Thus it is enough to show that the topology on $R^{r,n}$ is that of a disjoint sum of $R^{r,n}_1$ and $R^{r,n}_1$. There exists an open subset of $\Qnrd$ containing $R^{r,n}_1$ and not intersecting $R^{r,n}_2$ given by the invariant condition
    \[ \dime\left(\sum\limits_{i=1}^n \imag(X_i)\right) \geq k+1, \]
    and there exists an open subset of $\Qnrd$ containing $R^{r,n}_2$ and not intersecting $R^{r,n}_1$ given by the invariant condition
    \[ \dime\left(\bigcap\limits_{i=1}^n\lker(X_i)\right) \leq k. \qedhere \]
\end{proof}
As before, this yields the following corollary when taking the colimit under the sequence of embeddings $\iota_n$.
\begin{Corollary}
    Let $d\geq 2k+1$ and $r\geq k+1$. The cohomology of the colimit $R^{r,\infty}$ is isomorphic to the direct sum $H^*(\Grass(r,k)) \otimes H^*(\Grass(k+1,\infty)) \oplus H^*(\Grass(r,k+1)) \otimes H^*(\Grass(k,\infty))$.
\end{Corollary}
As noted before, the product structure of $\Qnrd$ with $\mathbb{A}^n$ is preserved under each of the embeddings, so the corollaries above also describe the cohomology of the loci in $\qin$.
\begin{Remark}
    Note that since in each of the cases above $\Zmax$ is an open subset of the smooth loci $R^{r,n}$, therefore $\Zmax$ is the subset of the smooth locus of $\qu$. In particular, by the results of Section \ref{sec1} the difference in dimension of $\qu$ and the dimension of its singular locus diverges to infinity as $n\rightarrow \infty$.
\end{Remark}

\section{\texorpdfstring{The cohomology of $\qu$ for $d \leq 2$}{sec4quot2}} \label{sec4quot2}

In this section we compute the cohomology of the Quot scheme $\qu$ for $d \leq 2$. The computation confirms the conjecture of Pandharipande (Question \ref{panda}) in the case $d=2$. Our method of proof involves applying the long exact sequence in cohomology for an abstract blowup (Theorem \ref{blowupseq}) to the resolution of singularieties of $\qudwa$ following from the results of \cite{S24}. Then a careful analysis using the Białynicki-Birula decomposition applied to the smooth scheme $\Hilb_2(\mathrm{Tot}(\mathbb{P}\mathcal{O}_{\mathbb{A}^n}^{\oplus  r}))$ yields the cohomology of the resolution and allows us to compute the result for $\qudwa$.\\

We begin by giving the easy result for the case $d=1$.
\begin{Proposition}
    Let $d=1$. The Quot scheme $\qu$ is isomorphic to $\aff^n \times \mathbb{P}^{r-1}$. In particular $H^*(\qu) \simeq \mathbb{Z}[x]/(x^{r})$.
\end{Proposition}
\begin{proof}
    Since each $X_j$ is uniquely determined by its trace, $\Qnrd$ introduced in Definition \ref{defQnrd} parametrizes kernels of $\kk$-linear surjections $(S^{\oplus r})_0 \twoheadrightarrow M$. Determining the kernel of $(S^{\oplus r})_0 \twoheadrightarrow M$ is equivalent to specifying a point of $\mathbb{P}^{r-1}$.
\end{proof}

\subsection{\texorpdfstring{The Białynicki-Birula decomposition of $\Hilb_2(\mathbb{A}^n \times \mathbb{P}^{r-1})$}{BB dla hilb2}}

In the rest of this section we study the case $d=2$. We start by proving that the the Hilbert scheme $\Hilb_2(\mathbb{A}^n \times \mathbb{P}^{r-1})$ is semiprojective (Definition \ref{defsemiprojective}), to show that we are able to emply our machinery of the Białynicki-Birula decomposition. Define the action of $\Gm$ on $\mathbb{A}^n \times \mathbb{P}^{r-1}$ with weights $\lambda_1,\dots,\lambda_n$ on the standard coordinates in $\mathbb{A}^n$ and weights $\gamma_1,\dots,\gamma_r$ on the coordinates in $\mathbb{P}^{r-1}$, assume furthermore that
\[ 0<\gamma_1 < \ldots < \gamma_r<\lambda_1 < \ldots < \lambda_n . \]
For any scheme $X$ with an action of $\Gm$ given by $\mu: \Gm \times X \rightarrow X$, a closed subscheme $Z \subset X$, and $t\in\Gm$ we define $t \cdot Z$ to be its image subscheme under the automorphism $\mu(t,-)$. The above induces the action $\mu: \Gm \times \Hilb_2(\mathbb{A}^n \times \mathbb{P}^{r-1}) \rightarrow \Hilb_2(\mathbb{A}^n \times \mathbb{P}^{r-1})$ of the torus on the Hilbert scheme $\Hilb_2(\mathbb{A}^n \times \mathbb{P}^{r-1})$.

\begin{Theorem}
    The Hilbert scheme $\Hilb_2(\mathbb{A}^n \times \mathbb{P}^{r-1})$ with the aforementioned toric action is a semiprojective variety.
\end{Theorem}
\begin{proof}
    Consider the inclusion
\begin{align*}
    \psi: \Hilb_2(\mathbb{A}^n \times \mathbb{P}^{r-1}) \hookrightarrow \Hilb_2(\mathbb{P}^n \times \mathbb{P}^{r-1})
\end{align*}
induced by the inclusion of $\mathbb{A}^n \times \mathbb{P}^{r-1}$ into $\Hilb_2(\mathbb{P}^n \times \mathbb{P}^{r-1})$. Define the action of $\Gm$ on $\mathbb{P}^n \times \mathbb{P}^{r-1}$ with weights $\lambda_0,\dots,\lambda_n$ and $\gamma_1,\dots,\gamma_r$ on the projective coordinates in $\mathbb{P}^n$ and $\mathbb{P}^{r-1}$ respectively, where $\lambda_0=0$ and the other weights are the same as previously. Abusing notation we denote the induced action on $\Hilb_2(\mathbb{P}^n \times \mathbb{P}^{r-1})$ also by $\mu$. Then the inclusion $\psi$ is $\Gm$-equivariant. For a point $[Z] \in \Hilb_2(\mathbb{P}^n \times \mathbb{P}^{r-1})$ we have
\begin{align*}
[Z] \in \Hilb_2(\mathbb{A}^n \times \mathbb{P}^{r-1})\ \text{ if and only if }\ \Supp(Z) \in \mathbb{A}^n \times \mathbb{P}^{r-1}.    
\end{align*}
Since $\lambda_0=0<\lambda_i$ for $i=1,\dots,n$, on the $\kk$-points $(x,y) \in \mathbb{P}^n \times \mathbb{P}^{r-1}$ we have that $\underset{t\rightarrow 0}{\limi}\ t\cdot (x,y) \notin \mathbb{A}^n \times \mathbb{P}^{r-1}$ if and only if $x\notin \mathbb{A}^n \times \mathbb{P}^{r-1}$. Since $\Supp(t\cdot [Z]) = t\cdot \Supp(Z)$ and $\Hilb_2(\mathbb{P}^n \times \mathbb{P}^{r-1})$ is projective (so the limit exists in $\Hilb_2(\mathbb{P}^n \times \mathbb{P}^{r-1})$) we see that all points of $\Hilb_2(\mathbb{A}^n \times \mathbb{P}^{r-1})$ have a limit at $t\in 0$, which lies in $(\Hilb_2(\mathbb{A}^n \times \mathbb{P}^{r-1}))^{\Gm}$.\\ 

We now show that the fixed locus $(\Hilb_2(\mathbb{A}^n \times \mathbb{P}^{r-1}))^{\Gm}$ is proper. Let us denote the coordinates on $\aff^n$ by $x_1,\dots,x_n$, and on $\mathbb{P}^{r-1}$ by $y_1,\dots,y_r$. This induces coordinates $y_{1/i},\dots,y_{r/i}$ on the affine open $U_i:=\{y_i \neq 0\}$. Then since
\[ \Supp(\underset{t\rightarrow0}{\limi}\ t\cdot Z) = \underset{t\rightarrow0}{\limi}\ t\cdot \Supp(Z) \]
the fixed points of the toric action above must consist of subschemes with support lying in
\begin{align*}
(\aff^n \times \mathbb{P}^{r-1})^{\Gm} = \bigcup\limits_{i=1}^r \{ 0 \times [e_i] \}.    
\end{align*}
If the subscheme is reduced, then it is just a union of two points, otherwise, since it has length $2$, it consists of a point in the fixed locus and an invariant tangent vector. The invariant tangent vectors at the points $0 \times [e_i]$ are exactly of the form $y_{j/i}^*$ for $j=1,\dots,r$ ($j\neq i$) or $x_k^*$ for $k=1,\dots,n$, where we denote by $x_k^*$ (respectively $y_{j/i}^*$) the dual vector in the distinguished basis of weight vectors. Thus the fixed locus $(\Hilb_2(\mathbb{A}^n \times \mathbb{P}^{r-1}))^{\Gm}$ is a finite number of points, so it is proper.
\end{proof}

We define a notion of support of a subscheme $Z \subset X$, which is used in our considerations of the $\Gm$-action on $\Hilb_2(\mathbb{A}^n \times \mathbb{P}^{r-1})$. More generally, it can be extended to any quasi-coherent sheaf $\mathcal{F}$ on $X$, with the definition below being the special case $\mathcal{F} = \mathcal{O}_Z$.
\begin{Definition}
    For a subscheme $Z \subset X$ we define $\Supp(Z)$ to be the locus of closed points $x$ of $X$, such that $(\mathcal{O}_Z)_x \neq 0$.
\end{Definition}

We describe the Białynicki-Birula decomposition of $\Hilb_2(\mathbb{A}^n \times \mathbb{P}^{r-1})$, identifying its positive and negative cells. We separate its fixed points into four groups.

\begin{Definition}
    The fixed points of the toric action on $\Hilb_2(\mathbb{A}^n \times \mathbb{P}^{r-1})$ are:
    \begin{enumerate}
        \item[(a)] $\{0 \times [e_i] \} \cup \{ 0 \times [e_j] \}$ for some $i<j$, $i,j\in [r]$.
        \item[(b)] Non-reduced with support $\{0 \times [e_i] \}$ and the tangent vector $y_{j/i}^*$ with $j>i$, $i,j\in [r]$.
        \item[(c)] Non-reduced with support $\{0 \times [e_i] \}$ and the tangent vector $y_{j/i}^*$ with $j<i$, $i,j\in [r]$.
        \item[(d)] Non-reduced with support $\{0 \times [e_i] \}$ and the tangent vector $x_k^*$ for some $k\in [n]$.
    \end{enumerate}
\end{Definition}

First we describe the positive cells of $\Hilb_2(\mathbb{A}^n \times \mathbb{P}^{r-1})$. We consider which
\[ [Z] \in \Hilb_2(\mathbb{A}^n \times \mathbb{P}^{r-1}) \]
have $\underset{t\rightarrow 0}{\limi}\ t\cdot [Z] = [Y]$ for some $[Y]$ in one of the above forms. Denote by $pr_1$, $pr_2$ the projections of $\aff^n \times \mathbb{P}^{r-1}$ to $\aff^n$ and $\mathbb{P}^{r-1}$ respectively.\\
Case 1. $[Y]$ of type $(a)$. Since
\begin{align*}
    \Supp(\underset{t\rightarrow 0}{\limi}\ t\cdot Z) = \underset{t\rightarrow 0}{\limi}\ t\cdot \Supp(Z),
\end{align*}
$[Z]$ must be a reduced scheme consisting of two points $[Z] = z_1 \cup z_2$, with
\begin{align*}
\underset{t\rightarrow 0}{\limi} \ t\cdot z_1  = \{ 0\times [e_i] \}\ \text{ and }\ \underset{t\rightarrow 0}{\limi} \ t\cdot z_2  = \{ 0\times [e_j] \}.
\end{align*}
For $i=1,2$ $pr_1(z_i)\in \aff^n$ can take arbitrary values, while $pr_2(x_i)$
 has to have first $i-1$ (respectively $j-1$) coordinates $0$ and the $i$-th one (respectively $j$-th one) nonzero. Thus they are parametrized by an affine space of dimension
 \begin{align*}
     (n+(r-i))+(n+(r-j)) = 2n+2r-i-j.
 \end{align*}
Case 2. $[Y]$ of type $(b)$. Since the loci of reduced subschemes is open in $\Hilb_2(\mathbb{A}^n \times \mathbb{P}^{r-1})$, if there exists at least one such scheme in some $X_i^+$, then their loci is full dimensional in $X_i^+$, thus it suffices to consider those for the dimension count. Consider $[Z] = z_1\cup z_2$ with $\underset{t\rightarrow 0}{\limi}\ t\cdot [Z]$ in $(b)$, we can assume that they lie on the same affine patch $\aff^n \times U_i$. Clearly the first $i-1$ coordinates of $pr_2(z_1),pr_2(z_2)$ have to be nonzero, and we can normalize the $i$-th to $1$. A limit of a linear form on this affine space is equal exactly to the lowest weight vector appearing with a nonzero coefficient in the expansion of this linear form in the basis $x_1^*,\dots,x_n^*,y_{1/i}^*,\dots,y_{r/i}^*$. Thus the line through $z_1,z_2$ under taking the limit with respect to the action of $\Gm$ is projected onto its coordinate with lowest weight. So the coordinates $i+1,\dots,j-1$ of $pr_2(z_1), pr_2(z_2)$ are equal. Thus the set of such $[Z]$ has dimension
\begin{align*}
    2n+ 2\cdot (r-j+1)+(j-i-1) = 2n+2r-i-j+1.
\end{align*}
Case 3. $[Y]$ of type $(c)$. Assume that a $[Z]$ with limit in $(c)$ is reduced, $[Z] = z_1 \cup z_2$. Then since $\Supp(\underset{t\rightarrow 0}{\limi}\ t\cdot Z) = \underset{t\rightarrow 0}{\limi}\ t\cdot \Supp(Z)$, $pr_2(z_1),pr(z_2)$ have the first $i-1$ coordinates zero. But then the equation $y_{s/i} = 0$ for $s<i$ vanishes on $[Z]$, so also in the limit, but this is clearly not the case. Thus $[Z]$ must be non-reduced. Then $pr_2(\Supp(Z))$ has first $i-1$ coordinates zero. As before, the tangent vector in the limit is projected onto its lowest weight coordinate, so its coefficients of $y_{1/i}^*,\dots,y_{j-1/i}^*$ have to be zero. The dimension of all such $[Z]$ is exactly 
\begin{align*}
    \underbrace{(n+r-i)}_{\text{choice of }\Supp(Z)} + \underbrace{(n+r-j-1)}_{\text{choice of the tangent}} = 2n+2r-i-j-1.
\end{align*}
Case 4. $[Y]$ of type $(d)$. As before it is enough to consider $[Z]$ reduced, $[Z] = z_1 \cup z_2$, and since the line joining $z_1$ with $z_2$ is projected onto its lowest weight nonzero coordinate, then we must have $pr_2(z_1) = pr_2(z_2) \in \mathbb{P}^{r-1}$, and the first $k-1$ coordinates of $pr_1(z_1), pr_1(z_2)$ have to be the same. Since
\[ \Supp(\underset{t\rightarrow 0}{\limi}\ t\cdot Z) = \underset{t\rightarrow 0}{\limi}\ t\cdot \Supp(Z),\]
the first $i-1$ coordinates of $pr_2(x_1), pr_2(x_2)$ have to be zero. It is easy to see that all such $z_1, z_2$ with the $k$-th coordinate of $pr_1(z_1), pr_1(z_2)$ different converge to $0 \times [e_i]$ with tangent vector $x_k^*$. Thus the dimension of all such $[Z]$ is exactly
\begin{align*}
  \underbrace{r-i}_{\text{choice of }pr_2(x_i)} + \underbrace{2\cdot (n-k+1)+ (k-1)}_{\text{choice of }pr_1(x_i)} = 2n + r -i -k+1.  
\end{align*}

While the decomposition above is interesting in its own terms (for example, it yields the number of closed points of $\Hilb_2(\mathbb{A}^n \times \mathbb{P}^{r-1})$ over finite fields), we still need to find the negative cells of the $\Gm$-action. This is because $\Hilb_2(\mathbb{A}^n \times \mathbb{P}^{r-1})$ is semiprojective and the negative cells of its Białynicki-Birula decomposition give us the cell decomposition of the core of $\Hilb_2(\mathbb{A}^n \times \mathbb{P}^{r-1})$. To achieve this we again consider the fixed points of the $4$ types listed above. First note, that for a point $[Z] \in \Hilb_2(\mathbb{P}^n \times \mathbb{P}^{r-1})$,
\begin{align*}
\underset{t\rightarrow \infty}{\limi}\ t\cdot [Z] \in \Hilb_2(\mathbb{A}^n \times \mathbb{P}^{r-1})\ \text{ if and only if }\ pr_1(\Supp(Z)) = 0, 
\end{align*}
as
\begin{align*}
[Z] \in \Hilb_2(\mathbb{A}^n \times \mathbb{P}^{r-1})\ \text{ if and only if }\ \Supp(Z) \in \mathbb{A}^n \times \mathbb{P}^{r-1}.    
\end{align*}
Let us now consider the loci of $[Z] \in \Hilb_2(\mathbb{A}^n \times \mathbb{P}^{r-1})$ converging to fixed points $[Y]$ of each of the four types when $t\rightarrow \infty$. The arguments are very similar to the case of positive cells (except with cases $(b)$ and $(c)$ reversed), but we sketch them nevetherless.\\
Case 1. $[Y]$ of type $(a)$. By the same argument as before $[Z]$ has to be reduced, $[Z] = z_1 \cup z_2$ for $pr_2(z_1)$ with the last $r-i$ coordinates zero, and $pr_2(z_2)$ with the last $r-j$ coordinates zero. Thus the dimension is
\begin{align*}
    (i-1)+(j-1) = i+j-2.
\end{align*}
Case 2. $[Y]$ of type $(b)$. By the same argument as in the case $(c)$ for positive cells we see that $[Z]$ must be non-reduced. $\Supp(Z)$ has to be of the form $0\times pr_2(\Supp(Z))$ with the last $r-i$ coordinates of $pr_2(\Supp(Z))$ zero. The image of the tangent vector under $pr_1$ must be zero, along with its last $r-i$ coordinates of its image under $pr_2$. Thus the dimension of all such $[Z]$ is
\begin{align*}
    \underbrace{(i-1)}_{\text{choice of }\Supp(Z)} + \underbrace{(j-1-1)}_{\text{choice of the tangent vector}} = i+j-3.
\end{align*}
Case 3. $[Y]$ of type $(c)$. As in the case of $(b)$ for the positive cells, for the dimension count it is sufficient to consider only reduced $[Z] = z_1 \cup z_2$. By assumption $pr_1(z_1)=pr_1(z_2)=0$ and since $\Supp(\underset{t\rightarrow 0}{\limi}\ t\cdot Z) = \underset{t\rightarrow 0}{\limi}\ t\cdot \Supp(Z)$, the last $r-i$ coordinates of $pr_2(z_1),pr_2(z_2)$ have to be zero. Furthermore, the $i$-th coordinate of $pr_2(z_1),pr_2(z_2)$ is nonzero, and thus we can consider $z_1,z_2$ in the affine chart $U_i$. Since the line connecting $z_1$ and $z_2$ is projected onto the highest weight coordinate, the coordinates $j+1,\dots,i-1$ in the affine chart $U_i$ have to be the same. All the points of this form with different $j$-th coordinate have a limit at the fixed point $[Y]$ in $(c)$. Thus the dimension of all such possible $[Z]$ is
\begin{align*}
    \underbrace{(i-1)}_{\text{choice of }z_1} + \underbrace{j}_{\text{choice of the first }j\text{ entries of }z_2} = i+j-1.
\end{align*}
Case 4. $[Y]$ of type $(d)$. Since points in $(d)$ don't satisfy all the equations $x_k=0$ for $k=1,\dots,n$, $[Z]$ has to be non-reduced. Clearly the last $r-i$ coordinates of $pr_2(\Supp(Z))$ have to be zero and the $i$-th coordinate has to be nonzero, thus it lies on the distinguished affine $U_i$. The tangent vector can be any vector in the tangent space to a point of $U_i$, whose highest weight coordinate with nonzero coefficient is $x_k^*$. Thus the dimension of the locus of such $[Z]$ is
\begin{align*}
    \underbrace{(i-1)}_{\text{choice of }pr_2(\Supp[Z])} + \underbrace{(r-1+k-1)}_{\text{choice of the tangent vector up to scaling}} = r+i+k-3.
\end{align*}

We see that the core of $\Hilb_2(\mathbb{A}^n \times \mathbb{P}^{r-1})$ is projective, as it is proper by Theorem \ref{coreproper} and is quasi-projective a closed subscheme of a quasi-projective scheme $\Hilb_2(\mathbb{A}^n \times \mathbb{P}^{r-1})$. It is in fact the Hilbert scheme of length two subschemes of a projective scheme $\mathrm{Spec}(k[x_1,\dots,x_n]/(x_1,\dots,x_n)^2) \times \mathbb{P}^{r-1}$. We can thus calculate its cohomology via cell decomposition. For a variety $X$ over $\mathbb{C}$, denote by
\[ H_X(t) := \sum\limits_{i=0}^{\infty} \mathrm{rk_{\mathbb{Z}}}(H^i(X,\mathbb{Z})  ) t^i \]
the Hilbert-Poincaré series of its cohomology ring, where $\mathrm{rk}_{\mathbb{Z}}$ denotes the rank of a $\mathbb{Z}$-module. Let $q=t^2$ be a formal variable of degree $2$. The computation of dimensions of the negative cells above yields the following: 

\begin{Theorem} \label{kohhilbb}
    Let $q=t^2$. The Hilbert-Poincaré series $H_{\Hilb}(q):=H_{\Hilb_2(\mathbb{A}^n \times \mathbb{P}^{r-1})}(q)$ of the Hilbert scheme $\Hilb_2(\mathbb{A}^n \times \mathbb{P}^{r-1})$ is equal to $H_a(q)+H_b(q)+H_c(q)+H_d(q)$, where
    \begin{enumerate}
        \item $H_a(q) = \underset{1\leq i <j \leq r}{\sum} q^{i+j-2}$,
        \item $H_b(q) = \underset{1\leq i <j \leq r}{\sum} q^{i+j-3}$,
        \item $H_b(q) = \underset{1\leq j < i \leq r}{\sum} q^{i+j-1}$,
        \item $H_d(q) = \underset{1\leq i \leq r}{\sum}\ \underset{1\leq k \leq n}{\sum} q^{r+i+k-3}$.
    \end{enumerate}
\end{Theorem}

We now use the above result to compute explicitly the Hilbert-Poincaré series of $H^*(\Hilb_2(\mathbb{A}^n \times \mathbb{P}^{r-1}))$. First, we have an equality
\begin{align*}
    \underset{1\leq i <j \leq r}{\sum} q^{i+j-2} = \frac{(1+q+\ldots+q^{r-1})^2-(1+q^2+\ldots+q^{2r-2})}{2}
\end{align*}
since each term $q^{i-1} \cdot q^{j-1}$ appears twice in the first product, and there are additional terms of the form $q^{i-1} \cdot q^{i-1}$ for $i,j\in [r]$. This can be further rewritten as
\begin{align*}
    \frac{(q^r-1)^2}{2(q-1)^2} - \frac{q^{2r}-1}{2(q^2-1)} =\frac{q^{2r}-q^{r+1}-q^r+q}{(q-1)^2(q+1)}.
\end{align*}
Summing up the contribution from the first three functions we get
\begin{align*}
    H_a(q)+H_b(q)+H_c(q) = \frac{(q^{2r-1}-q^{r}-q^{r-1}+1)\cdot (1+q+q^2)}{(q-1)^2(q+1)}.
\end{align*}
Second, we have a product factorization for $H_d(q)$
\begin{align*}
    \underset{1\leq i \leq r}{\sum}\underset{1\leq k \leq n}{\sum} q^{r+i+k-3} = q^{r-1} \cdot \Big( \underset{1\leq i \leq r}{\sum}q^{i-1} \Big) \cdot \Big(\underset{1\leq k \leq n}{\sum} q^{k-1} \Big) = q^{r-1} \cdot \frac{q^{r}-1}{q-1} \cdot \frac{q^n-1}{q-1}.
\end{align*}
Combining the two formulas, we get
\begin{align*}
    H_{\Hilb}(q) = \frac{(q^{2r-1}-q^{r}-q^{r-1}+1)\cdot (1+q+q^2)+q^{r-1}\cdot(q^{r}-1)\cdot (q^n-1)}{(q-1)^2(q+1)}.
\end{align*}
We can simplify this fraction to get
\begin{align} \label{formulahilb}
    H_{\Hilb}(q) = \frac{(q^r-1)\cdot (q^{n+r}+q^{n+r-1}+q^{r+1}-q^2-q-1)}{(q-1)^2(q+1)}.
\end{align}

\subsection{\texorpdfstring{Application of the resolution of singularieties of $\qudwa$}{resolution of sing}}
We now examine more carefully the resolution of singularities of $\qudwa$ of Theorem \ref{resolution}. Recall that for any coherent sheaf $\mathcal{E}$ on $\mathbb{A}^n$ we have a morphism of schemes
\begin{align*}
    \rho^2: \Hilb_2(\mathrm{Tot}(\mathbb{P}(\mathcal{E}))) \rightarrow \mathrm{Quot}_{\mathbb{A}^n}^{2}(\mathcal{E})
\end{align*}
induced by mapping a length two subscheme $Z$ of $\mathrm{Tot}(\mathbb{P}(\mathcal{E}))$ to the quotient
\begin{align*}
    \mathcal{E} = p_*\mathcal{O}_{\mathbb{P}(\mathcal{E})}(1) \rightarrow p_*\mathcal{O}_Z(1).
\end{align*}
that is a resolution of singularities. In our case $\mathcal{E} = \mathcal{O}_{\aff^n}^{\oplus r}$, thus $\mathrm{Tot}(\mathbb{P}(\mathcal{E})) \simeq \aff^n \times \mathbb{P}^{r-1}$. We would like to show that we are able to apply the sequence in Theorem \ref{blowupseq} to the morphism above. For this we need to understand the center $Z$ and the exceptional divisor
\[ Z' = Z\times_{\qudwa} \Hilb_2(\aff^n \times \mathbb{P}^{r-1}). \]
\begin{Theorem}
    The singular locus of $\qudwa$ is given by the closed subset where each $X_i$ for $i=1,\dots,n$ acts by a scalar.
\end{Theorem}
\begin{proof}
    Since the Quot scheme $\qudwa$ is irreducible \cite{JS22} and a simple calculation shows that it has dimension $2n+2r-2$, it is enough to show that the tangent space at a point has dimension $2n+2r-2$ if and only if this point does not lie in the locus above. We prove the claim by induction on $n$. The base case $n=2$ follows from Theorem \ref{gladkilocusqu}, since for $d=2$, $h^0(\mathbb{A}^2,\mathcal{E}nd(\mathcal{Q}_p)) \geq 2$, and a strict inequality occurs if and only if each $X_i$ for $i=1,2$ acts by a scalar. Assume $n>2$ and pick an element
    \[ \phi := [S^{\oplus r} \twoheadrightarrow M] \in \qudwa \]
    such that not all $X_i$ are multiples of the identity. By Theorem \ref{wymdwamacierze} we can assume, after a suitable change of basis, that $X_n$ acts by zero on $M$. Denote by $\tilde{\phi}$ the point $[S^{\oplus r}/(x_nS^{\oplus r}) \twoheadrightarrow M]$, considered as a point of $\quminus$. Denote $K:= \lker(\phi)$ and $\tilde{K} := \lker(\tilde{\phi})$. By induction
    \[ \dime(\Hom_{\kk[x_1,\dots,x_{n-1}]}(\tilde{K},M)) = \dime(\mathcal{T}_{\tilde{\phi}}\quminus) = 2 \cdot (n-1) + 2r -2.\]
    Let us consider $\mathcal{T}_{\phi}\qu = \Hom_S(K,M)$. Note that as a vector space $K = \tilde{K} + x_nS^{\oplus r}$. Specifying an element of $\Hom_S(K,M) = \Hom_S(K/x_nK,M)$
    is equivalent to specifying an element of $\Hom_{\kk[x_1,\dots,x_{n-1}]}(\tilde{K},M)$ and an element of $\Hom_{S}(x_nS^{\oplus r},M)$ that is zero on $x_nK$. The latter is equivalent to specyfing a homomorphism \[ \Hom_{\kk[x_1,\dots,x_{n-1}]}(\kk[x_1,\dots,x_{n-1}]^{\oplus r}/\tilde{K},M) = \Hom_{\kk[x_1,\dots,x_{n-1}]}(M,M). \] This space is exactly $h^0(\aff^{n-1},\mathcal{E}nd(\mathcal{Q}_{\tilde{\phi}}))$, and by Theorem \ref{wymdwamacierze} it has dimension $2$. So 
    \begin{gather*}
        \dime(\mathcal{T}_{\phi}\qu) = \dime(\Hom_S(K,M)) = \dime(\Hom_{\kk[x_1,\dots,x_{n-1}]}(\tilde{K},M)) + 2 = \\ = (2 \cdot (n-1) + 2r -2 )+ 2 = 2n+2r-2.    
    \end{gather*}
    The same inductive reasoning applied to $\phi$ for which all $X_i$ for $i=1,\dots, n$ act by scalar multiplication can be used to prove that
    \[ \dime(\mathcal{T}_{\phi}\qu) > 2n+2r-2. \]
    The distinction lies in that we have the inequality $\dime(\mathcal{T}_{\phi}\qu) > 2r+2$ in the base case $n=2$. Furthermore in the computation of the induction step we have
    \[ \dime(\Hom_{\kk[x_1,\dots,x_{n-1}]}(M,M)) = 4>2 , \]
    since all the vector space endomorphisms of $M$ commute with scalar multiplication.
\end{proof}
\begin{Definition}
    Let $Z$ be the singular locus of $\qudwa$. More specifically, by the theorem above, $Z$ is the closed subset defined by the condition
    \[ X_i = c_i \cdot  Id \ \text{ for all } i=1,\dots,n, \text{ some } c_i \in \kk . \]
\end{Definition}
By the product decomposition of $\qu$ in Lemma \ref{rozkladbezsladowe} we clearly have
\begin{align} \label{postacz}
    Z \simeq \aff^n \times \Grass(r,2),
\end{align}
since a morphism $S^{\oplus r} \twoheadrightarrow V$ for a length-two module $V$ with $S_{\geq 1}$ acting by $0$ on $V$ is given exactly by specifying a surjection $(S^{\oplus r})_0 \twoheadrightarrow V$.
\begin{Lemma} \label{ciagdziala}
    The morphism $\rho^2$ above is an abstract blowup with center $Z$.
\end{Lemma}
\begin{proof}
    Since the maximal dimension of a space of endomorphism of a module of length $2$ is $2$, see Theorem \ref{wymdwamacierze}, by Theorem \ref{gladkilocusqu} $Z$ is exactly the singular locus of $\qudwa$. So $\rho^{[2]}$ is an isomorphism away from $Z$ and $(\Hilba \backslash Z')_{\mathrm{red}} \simeq (\qudwa \backslash Z)_{\mathrm{red}}$ for
    \[ Z' = Z\times_{\qudwa} \Hilb_2(\aff^n \times \mathbb{P}^{r-1}). \qedhere \]
\end{proof}
To apply the sequence Theorem \ref{blowupseq} we need to understand the structure of the exceptional divisor $Z'$.
\begin{Lemma}
    The resolution $\rho^{[2]}$ restricted to the exceptional divisor
    \[ Z' = Z\times_{\qudwa} \Hilb_2(\aff^n \times \mathbb{P}^{r-1})\]
    is a projective bundle over $Z$ with fiber $\mathbb{P}^2$.
\end{Lemma}
\begin{proof}
    Consider an open subset $Z_K \subset Z$, where a fixed two-dimensional subspace $K \subset S_0^{\oplus r}$ surjects onto the module. Identifying $\mathbb{P}^{r-1}$ with $\mathbb{P}(S_0^{\oplus r})$ we see that projecting a length two subscheme in $Z'$ (that is, mapping to $Z$ under $\rho^{[2]}$) of $\aff^n \times \mathbb{P}^{r-1}$  to $K$ yields an element of $\Hilb_2(\mathbb{P}^1)$ and the preimage of such an element is uniquely determined by a point of $\aff^n \times \Grass(r,2) \simeq Z$. Thus over the open subset $Z_k$, $Z'$ has a product structure of $Z_k\times \Hilb_2(\mathbb{P}^1)$ with $\rho^{[2]}$ given by the projection to the second coordinate. Finally, $\Hilb_2(\mathbb{P}^1) \simeq \mathbb{P}^2$.
\end{proof}

Thus by the projective bundle formula we see that
\begin{align*}
    H^*(Z') \simeq H^*(\mathbb{P}^2) \otimes H^*(Z) \simeq \mathbb{Z}[q]/(q^3) \otimes H^*(Z).
\end{align*}
Therefore, letting $H_{Z'}(q)$ be the Hilbert-Poincaré series of $H^*(Z')$, which is well defined as a power series in $q$, since $\Grass(r,2)$ has cohomology concentrated in even degrees, we see that
\begin{align*}
    H_{Z'}(q) = H_{\Grass(r,2)}(q) \cdot H_{\mathbb{P}^2}(q) = H_{\Grass(r,2)}(q) \cdot (1+q+q^2).
\end{align*}
By the Formula \ref{postacz} and the Künneth formula we have an isomorphism between the cohomology of $Z$ and the cohomology of the Grassmanian of lines
\begin{align*}
    H^*(Z) \simeq H^*(\Grass(r,2)).
\end{align*}
By Lemma \ref{ciagdziala} we can apply the sequence in Theorem \ref{blowupseq} to the abstract blowup $Y\rightarrow X$ with center $Z$, and we obtain
\begin{gather*}
    \ldots \rightarrow H^i(X) \rightarrow H^i(Y) \oplus H^{i}(Z) \rightarrow H^i(Z') \rightarrow \\
    \rightarrow H^{i+1}(X) \rightarrow H^{i+1}(Y) \oplus H^{i+1}(Z) \rightarrow H^{i+1}(Z') \rightarrow \dots
\end{gather*}
By the above isomorphisms this sequence for $\rho^{[2]}$ is of the form
\begin{gather*}
    \ldots \rightarrow H^i(\qudwa) \rightarrow H^i(\Hilba) \oplus H^{i}(\Grass(r,2)) \rightarrow H^i(\Grass(r,2)) \times \mathbb{P}^2) \rightarrow \\
    \rightarrow H^{i+1}(\qudwa) \rightarrow H^{i+1}(\Hilba) \oplus H^{i+1}(\Grass(r,2)) \rightarrow H^{i+1}(\Grass(r,2)\times \mathbb{P}^2) \rightarrow \dots
\end{gather*}
By Theorem \ref{kohhilbb}, the cohomology of $\Hilba$ is concentrated in even degrees, and the same is true for the Grassmanian and $\mathbb{P}^2$. Thus for the sequence to be exact, $\qudwa$ also has to have cohomology concentrated in even degrees. Furthermore, this implies that the long exact sequence splits into short exact sequences
\begin{align*}
    0 \rightarrow H^i(\qudwa) \rightarrow H^i(\Hilba) \oplus H^{i}(\Grass(r,2)) \rightarrow H^i(\Grass(r,2)) \times \mathbb{P}^2) \rightarrow 0
\end{align*}
Let
\[ H_{\qudwa}(q) = H_{\qudwa}(t^2) \]
be the Hilbert-Poincaré series of $\qudwa$. Combined, the sequences above imply that
\begin{gather}
    H_{\qudwa}(q) = H_{\Hilb}(q) + H_{Z}(q) - H_{Z'}(q)
    = H_{\Hilb}(q) + H_{\Grass(r,2)}(q) - H_{\Grass(r,2)}(q) \cdot (1+q+q^2). \label{formulaqu}
\end{gather}
It is well know that the Hilbert-Poincaré series of $\Grass(r,2)$ is of the form
\begin{align*}
    H_{\Grass(r,2)}(q) = \frac{(1+q+\ldots+q^{r-1}) \cdot (1+q+\ldots+q^{r-2})}{1+q} = \frac{(q^r-1)\cdot (q^{r-1}-1)}{(q-1)^2\cdot (q+1)}
\end{align*}
Combining the above with Equation \ref{formulaqu} and Equation \ref{formulahilb} we get
\begin{align*}
    H_{\qudwa}(q) = \frac{(q^r-1)\cdot (q^{n+r}+q^{n+r-1}+q^{r+1}-q^2-q-1)}{(q-1)^2(q+1)} + \\ +\frac{(q^r-1)\cdot (q^{r-1}-1)}{(q-1)^2\cdot (q+1)} - \frac{(q^r-1)\cdot (q^{r-1}-1)}{(q-1)^2\cdot (q+1)} \cdot (1+q+q^2),
\end{align*}
that is
\begin{align*}
    H_{\qudwa}(q) = \frac{(q^r-1)\cdot (q^{n+r}+q^{n+r-1}+q^{r+1}-q^2-q-1) - (q+q^2) \cdot (q^r-1)\cdot (q^{r-1}-1)}{(q-1)^2(q+1)}.
\end{align*}
Reducing the terms on the right-hand side we get
\begin{align*}
    H_{\qudwa}(q) = \frac{(q^r-1)\cdot (q^{n+r}+q^{+r-1}-q^r-1)}{(q-1)^2(q+1)}.
\end{align*}
Grouping together the terms divisible by $q^n$ we get
\begin{align*}
    H_{\qudwa}(q) = \frac{1-q^{2r}}{(1-q^2)(1-q)} + q^{n+r-1} \cdot \frac{q^r-1}{(1-q^2)(1-q)}.
\end{align*}
Since the cohomology groups $H^k$ commute with filtered colimits, this implies that
\begin{align*}
    H_{\mathrm{Quot}_2(\mathcal{O}_{\aff^{\infty}}^{\oplus r})} = \frac{1-q^{2r}}{(1-q^2)(1-q)}.
\end{align*}
The above answers in the affirmative the Question \ref{panda} of Pandharipande, since the Hilbert series of the ring $\kk[x_1,x_2]/(x_2^r)$, where the variable $x_i$ lies in degree $2i$ for $i=1,2$ is equal to
\begin{align*}
    (1+q+q^2+\ldots) \cdot (1+q^2+ \dots + q^{2 (r-1)}) = \frac{1-q^{2r}}{(1-q)^2(1+q)}.
\end{align*}
In fact, we have proved the following about the cohomology of $\qudwa$.
\begin{Theorem}
    The Hilbert-Poincaré series of $H^*(\qudwa)$ agrees with the Hilbert-Poincaré series of $\mathbb{Z}[x_1,x_2]/(x_2^r)$ in degrees $\leq 2\cdot (n+r-2)$.
\end{Theorem}
Note that this bound actually agrees with the dimension of the Quot scheme $\qudwa$. It would be interesting to see if the equality of Hilbert-Poincaré series of $H^*\big(\mathrm{Quot}_2(\mathcal{O}_{\aff^{\infty}}^{\oplus r})\big)$ and $\mathbb{Z}[x_1,x_2]/(x_2^r)$ can be upgraded to an isomorphism of rings.
\bibliographystyle{alpha}
\bibliography{bibquot2}

\end{document}